\documentclass[11pt, a4paper, fleqn]{article}
\usepackage[latin1,utf8]{inputenc}
\usepackage{amsmath}
\usepackage{amsthm}
\usepackage{amsfonts}
\usepackage{amssymb}
\usepackage[T1]{fontenc}
\usepackage[all]{xy}
\usepackage{bbm}
\usepackage{enumitem}
\usepackage{color}
\usepackage[hidelinks]{hyperref}
\usepackage{graphicx}
\usepackage{transparent}
\usepackage{graphics}
\usepackage{xcolor}
\usepackage{color}
\usepackage{geometry}
\usepackage{marginnote}

\usepackage{pgf,tikz}
\usepackage{mathrsfs}
\usetikzlibrary{arrows,decorations.pathmorphing,backgrounds,positioning,fit,petri,cd}

\theoremstyle{plain} 
\newtheorem{theorem}{Theorem}[section]
\newtheorem{proposition}[theorem]{Proposition}
\newtheorem{lemma}[theorem]{Lemma}
\newtheorem{corollary}[theorem]{Corollary}

\theoremstyle{remark}
\newtheorem{remark}[theorem]{Remark}
\newtheorem{example}[theorem]{Example}

\theoremstyle{definition}
\newtheorem{definition}[theorem]{Definition}

\newcommand{\rep}[1]{%
  {%
    \tiny%
    \begin{matrix}%
      #1%
    \end{matrix}%
  }%
}

\def\Hom{\operatorname{Hom}}

\def\coker{\operatorname{coker}}
\def\im{\operatorname{Im}}
\def\Fac{\operatorname{Fac}\,}

\def\add{\operatorname{add}\,}
\def\mod{\operatorname{mod}}
\newcommand{\A}{\mathcal{A}}

\newcommand{\X}{\mathcal{X}}
\newcommand{\Filt}{\operatorname{Filt}}

\newcommand{\I}{\mathcal{I}}

\newcommand{\wH}{\widehat{\mathscr{H}}}
\newcommand{\J}{\mathcal{J}}
\newcommand{\F}{\mathcal{F}}
\renewcommand{\P}{\mathcal{P}}
\newcommand{\T}{\mathcal{T}}

\newcommand{\CT}{\mathfrak{T}}
\newcommand{\Stab}{\operatorname{Stab}}

\usepackage[nottoc]{tocbibind}

\title{An algebraic approach to Harder-Narasimhan filtrations}

\author{Hipolito Treffinger}

\begin{document}

\title{An algebraic approach to Harder-Narasimhan filtrations}

\author{Hipolito Treffinger}
\date{}
\maketitle


\abstract{
In this article we study chains of torsion classes in an abelian category $\mathcal{A}$.
We prove that chains of torsion classes satisfying mild technical conditions induce a Harder-Narasimhan filtration for every non-zero object $M$ in $\mathcal{A}$, generalising a well-known property of stability conditions.
We also characterise the slicings of $\mathcal{A}$ in terms of chains of torsion classes. 
We finish the paper by showing that chains of torsion classes induce wall-crossing formulas in the completed Hall algebra of the category.}


\section{Introduction}

The study of stability conditions started with the introduction of Geometric Invariant Theory in \cite{Mumford1965} and quickly became an important tool in the study of moduli spaces in algebraic geometry. 
Inspired by problems of a similar nature arising in other areas of mathematics, many authors adapted the notion of stability conditions to different contexts such as quiver representations \cite{Schofield1991}, abelian categories \cite{King1994, Rudakov1997}, and triangulated categories\;\cite{Bridgeland2007}.

The philosophy behind stability conditions is that every stability condition determines classes of distinguished objects which behave better than the original category, which are known as \textit{semistable objects}.
A key property of stability conditions is that every object in the category admits a filtration by semistable objects that is unique up to isomorphism. 
This filtration is known as the \textit{Harder-Narasimhan} filtration in honour of Harder and Narasimhan, whose work described this phenomena for the first time in \cite{Harder1975}.

In this paper we study the notions of Harder-Narasimhan filtrations from a torsion theoretic point of view, showing the existence of Harder-Narasimhan filtrations that are not induced by stability conditions, see Example~\ref{ex:nowide} and Remark~\ref{rmk:notallchains}.

From now on we fix $\A$ to be an abelian category.
We note that most of our results hold for triangulated categories if the language is changed accordingly. 
A key notion in the study of abelian categories is that of a torsion pair, first introduced in \cite{Dickson1966}.
This concept generalises to any abelian category the fact that every abelian group of finite rank is isomorphic to the direct sum of a torsion and a torsion free group.
A torsion pair is a pair $(\T, \F)$ of full subcategories in $\A$ such that $\Hom_\A (X,Y)=0$ for all $X \in \T$ and $Y \in \F$ and for every object $M$ in $\A$ there exists a short exact sequence
$$0\to tM\to M\to M/tM\to 0$$
where $tM\in\mathcal{T}$ and $M/tM\in\mathcal{F}$.
It follows from this definition that this short exact sequence, known as the \textit{canonical short exact sequence} of $M$ with respect to $(\T, \F)$, is unique up to isomorphism.
If $(\T, \F)$ is a torsion pair we say that $\T$ is a torsion class and that $\F$ is a torsion free class.
We now introduce the main object of study of this paper.

\begin{definition}
Let $\A$ be an abelian category and consider the closed set $[0,1]$ of the real numbers.
A chain of torsion classes $\eta$ is a set of torsion classes in $\A$ indexed by the interval $[0,1]$ as follows.
$$\eta= \{\T_s : s\in[0,1], \T_0=\A, \T_1=\{0\} \text{ and } \T_s \subseteq \T_r \text{ if } r\leq s\}.$$
\end{definition}

Given a chain of torsion classes $\eta$, for each $t\in[0,1]$, we define a subcategory $\P_\eta(t)$ of $\eta$-quasisemistable objects of phase $t$. 
We denote $\P_\eta$ the set $\{\P_\eta(t) \colon t\in [0,1]\}$.
For the full definition of quasisemistable objects please see Definition~\ref{def:quasisemistables}.
Even if these categories can be defined for every chain of torsion classes, one needs to put some mild assumptions on them for $\P_\eta$ to be well-behaved. 
We denote by $\CT(\A)$ the set of all chains of torsion classes satisfying the needed assumptions, see Definition\;\ref{def:CT(A)} for details.
In our first result we show that every chain of torsion classes $\eta \in \CT(\A)$ induces a Harder-Narasimhan filtration on every non-zero object $M$ of\;$\A$. 

\begin{theorem}[Theorem\;\ref{thm:algHNfilt}]
Let $\A$ be an abelian category and $\eta$ be a chain of torsion classes in $\CT(\A)$.
Then every non-zero object $M$ of $\A$ admits a Harder-Narasimhan filtration.
That is a filtration $M_0 \subset M_1 \subset \dots \subset M_n$
such that:
\begin{enumerate}
\item $0 = M_0$ and $M_n=M$;
\item there exists $r_k \in [0,1]$ such that $M_k/M_{k-1} \in \P_\eta(r_k)$ for all $1\leq k \leq n$;
\item $r_1 > r_2 > \dots > r_n$.
\end{enumerate}
Moreover this filtration is unique up to isomorphism.
\end{theorem}

A well-known property of stability conditions is that every stability condition $\phi$ defined over $\A$ induces two chains of torsion classes $\eta_\phi$ and $\overline{\eta}_\phi$ in $\A$, see Lemma~\ref{lem:stab-chain}.
However not every chain of torsion classes in $\A$ is induced by a stability condition, see Remark~\ref{rmk:notallchains}. 
If $\eta_\phi$ is the chain of torsion classes associated to a stability condition $\phi$ we show that the Harder-Narasimhan filtrations induced by $\phi$,  $\eta_\phi$ and $\overline{\eta}_\phi$ coincide. 

\begin{theorem}[Theorem\;\ref{thm:semistabchain}]
Let $\phi:Obj^*(\mathcal{A})\rightarrow [0,1]$ be a stability function, $\eta_\phi$ and $\overline{\eta}_\phi$ be the chains of torsion classes induced by $\phi$ and $t \in [0,1]$.
Then 
$$\P_{\eta_\phi}(t) = \P_{\overline{\eta}_\phi}(t) = \{M \in \A: M \text{ is $\phi$-semistable and } \phi(M)=t\} \cup \{0\}.$$
In particular the Harder-Narasimhan filtrations of $M$ induced by $\phi$,  $\eta_\phi$ and $\overline{\eta}_\phi$ coincide for all non-zero $M \in \A$.
\end{theorem}

The main difficulty in defining stability conditions in a non-abelian setting is that given an object $M$, in general there are not enough proper subobjects of $M$ to make the stability conditions interesting. 
For instance, in the case of a triangulated category the proper subobjects of a given object $M$ are its direct summands.
Hence, the naive definition of a stability condition on a triangulated category would imply that every indecomposable object $M$ is stable, for every stability condition.
The breakthrough in this case was done in \cite{Bridgeland2007}, where it was realised that a much richer set of stability conditions could be defined introducing the concept of \textit{slicing}.
In Section~\ref{sec:slicing} we characterise the slicings of abelian categories in terms of chains of torsion classes.

\begin{theorem}[Theorem\;\ref{thm:eqslicingtors}]
Let $\A$ be an abelian category.
Then every chain of torsion classes $\eta$ in $\CT(\A)$ induces a slicing $\P_\eta$. 
Moreover, every slicing $\P$ of $\A$ arises this way.
\end{theorem}
It is worth noting that in \cite{TattarTreffinger} it has been shown that every slicing is induced by a weak stability condition, a generalisation of stability conditions introduced in \cite{Joyce2007}.

A very important class of chain of torsion classes is the class of \textit{maximal green sequences}, see Section~\ref{sec:MGS}. 
Although they can be defined in arbitrary abelian categories, they are particularly prominent in when $\A = \mod A$, the category of finitely generated modules over a finite-dimensional algebra $A$.
Combining the results of this paper together with those of \cite{SchrollTreffinger, Treffinger2019} we obtain the following. 
We note that an equivalent result was proven independently by Demonet in \cite{DK2019}.

\begin{theorem}[Theorem~\ref{thm:MGSforAlg}]
Let $A$ be a finite-dimensional algebra. 
Then every maximal green sequence of length $l$ in $\mod A$ is determined by a set $\{B_1, \dots, B_l\}$ of bricks in $\mod A$ satisfying that $\Hom_A(B_i, B_j)=0$ if $i > j$.
\end{theorem}

An important application of stability conditions was found in \cite{Reineke2003}, where it was shown that any stability condition over an abelian category $\A$ induces a factorisation of a distinguished element $e_\A$ in the completed Hall algebra $\wH(\A)$ of $\A$, giving rise to examples of wall-crossing formulas.
This result can be generalised as a consequence of Theorem~1.2. 
Indeed, we show that every chain of torsion classes $\eta \in \CT(\A)$ gives rise to a distinguished factorisation of $e_\A$, see Theorem~\ref{thm:wallcrossing}.

\bigskip
The structure of the paper is the following. 
In Section\;\ref{sec:HNfilt} we define chains of torsion classes and we study some of their basic properties, including the existence of Harder-Narasimhan filtrations for every non-zero object in the category.
In Section\;\ref{sec:ordering} we show that every chain of torsion classes induces a preorder on the abelian category and we study the behaviour of the quasisemistable objects under this preorder.
Section\;\ref{sec:stab} is dedicated to the comparison between chains of torsion classes and stability conditions.
In Section\;\ref{sec:slicing} we characterise slicings in terms of chain of torsion classes. 
In Section\;\ref{sec:MGS} we study some consequences of our results in the module categories of finite-dimensional algebras.
Finally, in Section\;\ref{sec:wallcrossing} we consider an application of chains of torsion classes on the study of Hall algebras.


\section{Filtrations induced by chains of torsion classes}\label{sec:HNfilt}

In this section introduce the set of chains of torsion classes $\CT(\A)$ of $\A$ and we prove that every chain of torsion classes in an abelian category $\A$ satisfying some mild assumptions induce a Harder-Narasimhan filtration on every non-zero object $M$ of $\A$. 

A torsion pair is a pair $(\T, \F)$ of full subcategories in $\A$ such that $\Hom_\A (X,Y)=0$ for all $X \in \T$ and $Y \in \F$ and for every object $M$ in $\A$ there exists the a short exact sequence
$$0\to tM\to M\to M/tM\to 0$$
where $tM\in\mathcal{T}$ and $M/tM\in\mathcal{F}$.
It follows from this definition that this short exact sequence, known as the \textit{canonical short exact sequence} of $M$ with respect to $(\T, \F)$, is unique up to isomorphism.A schematic picture of a torsion pair is depicted in Figure~\ref{fig:torsionpair}.

\begin{figure}
\begin{center}
    \begin{tikzpicture}
    \begin{scope}
    \clip (0,0) ellipse (4 and 2);
    \draw[fill=gray!20] (5,0) ellipse (4 and 3);
    \draw[fill=gray!30] (-5,0) ellipse (4 and 3);
    \draw[thick] (0,0) ellipse (4cm and 2cm);
    \end{scope}
    \node[] at (-2.3, 0)  {$\F$};
    \node[] at (2.3, 0)  {$\T$};
    \node[anchor=north west] at (3, 2)  {$\A$};
\end{tikzpicture}
\end{center}
    \caption{A torsion pair $(\T, \F)$ in $\A$.}
    \label{fig:torsionpair}
\end{figure}

We recall the definition of a chain of torsion classes given in the introduction. 

\begin{definition}\label{def:indexchaintorsion}
Let $\A$ be an abelian category and consider the closed set $[0,1]$ of the real numbers.
A chain of torsion classes $\eta$ is a set of torsion classes indexed by the interval $[0,1]$ such that $\T_0 = \A$, $\T_1 = \{0\}$ and $\T_s \subseteq \T_r$ if $r \leq s$.
$$\eta:= \{\T_s : s\in [0,1] \text{ and } \T_0=\A, \T_1=\{0\} \text{ and } \T_s \subseteq \T_r \text{ if } r\leq s\}.$$
\end{definition}

\begin{remark}
It follows directly from the definition that given a chain of torsion classes $\eta$ and two distinct elements $s,t \in [0,1]$ we might have $\T_s = \T_r$, as it can be seen in Examples~\ref{ex:torsionpair} and \ref{ex:nowide}.
\end{remark}

\begin{figure}
\begin{center}
    \begin{tikzpicture}
    \begin{scope}
    \clip (0,0) ellipse (4 and 2);
    \draw[fill=gray!10] (1.5,0) ellipse (4 and 3);
    \draw[fill=gray!20] (2.5,0) ellipse (4 and 3);
    \draw[fill=gray!30] (5,0) ellipse (4 and 3);
    \draw[fill=gray!40] (6,0) ellipse (4 and 3);
    \draw[ultra thick] (0,0) ellipse (4cm and 2cm);
    \end{scope}
    \node[] at (0.6, 0)  {$\dots$};
    \node[] at (3.3, 0)  {$\dots$};
    \node[anchor=west] at (-4, 0)  {$\T_0=\A$};
    \node[anchor=west] at (-2.5, 0)  {$\T_{a_1}$};
    \node[anchor=west] at (-1.5, 0)  {$\T_{a_2}=\T_{a_3}$};
    \node[anchor=west] at (1, 0)  {$\T_{a_4}$};
    \node[anchor=west] at (2, 0)  {$\T_{a_5}$};
    \filldraw[black] (4,0) circle (2pt);
    \node[anchor=west] at (4, 0)  {$\T_1=\{0\}$};
    \node[anchor=north] at (-4, -3) (0) {$0$};
    \node[anchor=north] at (4, -3) (1) {$1$};
    \draw[thick] (-4,-3) -- (4,-3);
    \foreach \y in {1,2,3,4,5}
    \draw[thick] (40*\y-4 cm, -3.1) -- (40*\y-4 cm, -2.9) node[anchor=north, below] {$a_\y$};
    \draw[thick] (-4 cm, -3.1) -- (-4 cm, -2.9);
    \draw[thick] (4 cm, -3.1) -- (4 cm, -2.9);
\end{tikzpicture}
\end{center}
    \caption{A chain of torsion classes $\eta$ in $\A$.}
    \label{fig:chaintorsionclasses}
\end{figure}

Also, every chain of torsion classes $\eta$ determines uniquely a chain of torsion free classes as follows. 
$$\{\F_s : s\in [0,1] \text{ and }(\T_s, \F_s) \text{ is a torsion pair for all } \T_s \in [0,1]\},$$
where it follows from the definitions that $\F_r \subseteq \F_s \text{ if } r\leq s$.
From now on, given a chain of torsion classes $\eta$ and  $s \in [0,1]$ we denote by $\F_s$ to be the torsion free class such that $(\T_s, \F_s)$ is a torsion pair in $\A$.

\smallskip
One of the main characteristics of torsion pairs is the existence of a canonical short exact sequence associated to any object in the category. 
It is easy to see that if two torsion classes $\T$ and $\T'$ are such that $\T \subseteq \T'$ then for every object $M$ of $\A$ we have that $tM$ is a subobject of $t'M$ and, moreover, that the torsion object $tt'M$ of $t'M$ with respect to $\T$ is isomorphic to $tM$.
In particular, if we take a chain of torsion classes $\eta$ and an interval $(a,b) \subset [0,1]$, we have that for every non-zero object $M$ of $\A$, we have a chain of subobjects 
\begin{equation}\label{eq:stabilisation}
    0 \subseteq \dots \subseteq t_{r_2} M \subseteq \dots \subseteq t_{r_1} M \subseteq \dots \subseteq M
\end{equation}
where $a < r_1 < r_2 < b$.
The existence of a Harder-Narasimhan filtration for $M$ relies on the stabilisation in both directions of chains of the form (\ref{eq:stabilisation}) for every object $M$ and every interval $(a,b)$.
We follow ideas of \cite{Rudakov1997} (\textit{cf.} Definition~\ref{def:StabA}) to introduce a necessary mild technical constraint on the chains that we consider in this paper.

\begin{definition}\label{def:CT(A)}
A chain of torsion classes $\eta$ in $\A$ is said to be \textit{quasi-noetherian} if for every subset $S\subset[0,1]$ and every object non-zero $M$ of $\A$ there exists $s(M)\in S$ such that the torsion object $t_{s(M)} M$ of $M$ with respect to $\T_{s(M)}$ contain the torsion object $t_r M$ of $M$ with respect to $\T_r$ for all $r \in S$.

Dually, a chain of torsion classes $\eta$ in $\A$ is said to be \textit{weakly-artinian} if for every subset $S\subset[0,1]$ and every non-zero object $M$ of $\A$ there exists $s'(M)\in S$ such that the torsion object $t_{s'(M)} M$ is in $t_r M$ for all $r \in S$.

We denote by $\CT(\A)$ the set of all chains of torsion classes in $\A$ that are quasi-noetherian and weakly-artinian. 
\end{definition}

\begin{remark}\label{rmk:etainCT(A)}
Suppose that $\eta$ is a chain of torsion classes for which there is a finite family of torsion classes $\{\X_1, \X_2, \dots, \X_n\}$ satisfying the following: for every $r \in [0,1]$ there is a $\X_i \in \{\X_1, \X_2, \dots, \X_n\}$ such that $\T_r = \X_i$.
Then $\eta$ belongs to $\CT(\A)$.

Also, if $\A$ is an artinian and noetherian category, then every chain of torsion classes $\eta$ is an element of $\CT(\A)$.
An important example of such categories is the category of finitely generated modules over a finite-dimensional algebra over a field.
\end{remark}

Below we include two explicit examples of a chain of torsion classes illustrating Definition~\ref{def:CT(A)}.

\begin{example}\label{ex:torsionpair}
Let $\A$ be an abelian category and consider a non-trivial torsion pair $(\T,\F)$ in $\A$. 
Then for every pair of real numbers $a,b \in [0,1]$ such that $a < b$ we can construct the chain of torsion classes $\eta_{(a,b)}$, where $\T_s \in \eta_{(a,b)}$ is defined as follows.
$$
\T_s= \left\{\begin{array}{ll}
\A & \text{ if $t \in [0,a)$}\\
\T & \text{ if $t \in [a,b)$}\\
\{0\} & \text{ if $t\in [b, 1]$}
\end{array}
\right.
$$
Then clearly $\eta_{(a,b)} \in\CT(\A)$. 
\end{example}

\begin{example}\label{ex:nowide}
Consider the path algebra $kQ$ of the following quiver.
$$\xymatrix{
 Q: \quad 1 \ar[r] & 2 \ar[r] & 3
}$$
It is well known that $A=kQ$ is a hereditary algebra and that $\mod A$ is an abelian length category having six non-isomorphic indecomposable modules that can be ordered in the so-called Auslander-Reiten quiver of $A$ as follows.

\begin{center}
	\begin{tikzcd}
		&	& \rep{1\\2\\3}\ar[dr]	&	&	\\
    	&\rep{2\\3}\ar[ur]\ar[dr]\arrow[rr, dotted, no head]	&	&\rep{1\\2}\ar[dr]	&	\\
    	\rep{3}\ar[ur]\arrow[rr, dotted, no head]&	&\rep{2}\ar[ur]\arrow[rr, dotted, no head]	&	&\rep{1}	
	\end{tikzcd}
\end{center}

We can build the chain of torsion classes $\eta \in \CT(\mod A)$ where $\T_t \in \eta$ is as follows.
$$
\T_t= \left\{\begin{array}{ll}
\mod A & \text{ if $t=0$}\\
\add\left\{\rep{1} \oplus \rep{2} \oplus \rep{2\\3} \oplus \rep{1\\2} \oplus \rep{1\\2\\3}\right\} & \text{ if $t \in (0,1/3]$}\\
\add\{\rep{1}\} & \text{ if $ t \in (1/3, 2/3]$}\\
\{0\} & \text{ if $t\in (2/3,1]$.}
\end{array}
\right.
$$

\end{example}

\begin{example}\label{ex:continousP}
Take the path algebra $\mathbb{C}Q$ over the complex numbers of the Kronecker quiver 
$$Q:\begin{tikzcd}1 \arrow[r, "\alpha", shift left = 1]\arrow[r, "\beta"', shift left = -1]
& 2\end{tikzcd}$$
and consider the category $\mod \mathbb{C}Q$ of right $\mathbb{C}Q$-modules.
Inside this category there is an infinite family of indecomposable modules $M(\lambda, n)$ parameterized by $\lambda \in \mathbb{P}^1(\mathbb{C})$ and $n \in \mathbb{Z}_{>0}$.
For every subset $S$ of $\mathbb{P}^1(\mathbb{C})$ there is a unique torsion class $\T_S$ which is the minimal torsion class containing all the modules $M(\lambda,n)$ with $\lambda\in S$.
Moreover, given $S$ and $R$ two subsets of $\mathbb{P}^1(\mathbb{C})$ we have that $\T_S \subseteq \T_R$ if and only if $S \subseteq R$.

Now, any bijection $f: \mathbb{P}^1(\mathbb{C}) \to \mathbb{S}^2$ between $\mathbb{P}^1(\mathbb{C})$ and the $2$-sphere 
$$\mathbb{S}^2= \{ (x,y,z) \in \mathbb{R}^3 : x^2 + y^2 + z^2 = 1\}$$
sending a $\lambda \in \mathbb{P}^1(\mathbb{C})$ to $f(\lambda)=(f_x(\lambda), f_y(\lambda), f_z(\lambda))$
induces a chain of torsion classes $\eta_f \in \CT(\mathbb{C}Q)$, where for every $r \in (0,1)$ we define $\T_r \in \eta_f$ to be $\T_r = \T_{S_r}$, where $S_r = \{ \lambda\in\mathbb{P}^1(\mathbb{C}):
f_z(\lambda) \geq r\}$.
Since $\eta_f$ is a chain in $\mod \mathbb{C}Q$, the remark above implies that $\eta_f \in \CT(\mod \mathbb{C}Q)$. 
\end{example}

\begin{proposition}\label{prop:limsup}
Let $\eta$ be a chain of torsion classes in $\CT(\A)$ and let $I=(r_1, r_2)$ be an interval in $[0,1]$.
Then $\bigcap\limits_{r\in I} \T_r$ and $\bigcup\limits_{r\in I} \T_r$
are torsion classes in $\A$, while $\bigcap\limits_{r\in I} \F_r$ and $\bigcup\limits_{r\in I} \F_r$ are torsion free classes in $\A$.
\end{proposition}

\begin{proof}
We only prove that $\bigcup\limits_{r\in I} \T_r$ is a torsion class, the other cases being similar.
Let $M\in \bigcup\limits_{r\in I} \T_r$ and $N$ be a quotient of $M$. 
Then $M\in \T_r$ for some $r\in I$.
Hence $N\in \T_r$ because $\T_r$ is closed under quotients. 
Therefore $N\in \bigcup\limits_{r\in I} \T_r$. 
In other words, $\bigcup\limits_{r\in I} \T_r$ is closed under quotients.

Let $M, M'\in \bigcup\limits_{r\in I} \T_r$. 
Then $M\in \T_s$ and $M'\in \T_{s'}$ for some $s, s' \in I$.
Without loss of generality we can suppose that $s \leq s'$.
Then $\T_{s'} \subseteq \T_s$ by Definition\;\ref{def:indexchaintorsion}.
In particular $M' \in \T_s$.
Then for every short exact sequence 
$$0\to M' \to E \to M \to 0$$
we have that $E$ belong to $\T_s$ because $\T_s$ is closed under extensions. 
Then $\bigcup\limits_{r\in I} \T_r$ is closed under quotients and extensions.

To finish the proof we need to show that every object $M \in \A$ has a torsion subobject $tM \in \bigcup\limits_{r\in I} \T_r$.
By hypothesis, $\eta \in \CT(\A)$.
In particular, $\eta$ is a quasi-noetherian, that is, there exists $s(M) \in I$ such that $t_r M$ is a subobject of $t_{s(M)} M$ for all $r \in I$.
We claim that $t_{s(M)} M$ is the torsion object of $M$ with respect to $\bigcup\limits_{r\in I} \T_r$.
Indeed, if $L$ is a subobject of $M$ such that $L \in\bigcup\limits_{r\in I} \T_r$, there exists a $r \in I$ such that $L \in \T_r$. 
Now, $L$ is a subobject of the torsion object $t_r M$ of $M$ with respect to $\T_r$ which, in turn, is a subobject of $t_s M$.
Hence $L$ is a subobject of $t_{s(M)} M$.
\end{proof}

\begin{example}
Consider $\eta_f$ as in Example~\ref{ex:continousP} and take an interval $I=(r_1, r_2)$ inside $[0,1]$. 
Then $T_r$ is strictly included in $T_s$ if $s < r$. 
A straightforward verification shows that $\bigcap\limits_{r\in I} \T_r =  \T_{r_2}$ where $r_2 \not \in (r_1, r_2)$. 
\end{example}

\begin{proposition}
Let $\eta \in \CT(\A)$ be a chain of torsion classes. 
Then for every $t\in [0,1]$ we have that 
$$\left( \bigcup_{r>t} \T_r ,  \bigcap_{s>t} \F_s
\right)$$
and 
$$\left( \bigcap_{r<t} \T_r ,  \bigcup_{s<t} \F_s
\right)$$
are torsion pairs in $\A$.
\end{proposition}

\begin{proof}
We show that 
$$\left( \bigcup_{r>t} \T_r ,  \bigcap_{s>t} \F_s
\right)$$
is a torsion pair for every $t \in [0,1]$, being the other proof very similar.

We have by Proposition \ref{prop:limsup} that $\bigcup\limits_{r>t} \T_r$ is a torsion class for all $t\in [0,1]$. 
Then there is a torsion free class $\F'$ in $\A$ such that $\left(\bigcup\limits_{r>t} \T_r , \F' \right)$ is a torsion pair
where $\F'=(\bigcup\limits_{r>t} \T_r)^\perp$.
We show that $\F' = \bigcap\limits_{s>t} \F_s$ by double inclusion. 
Let $Y \in \bigcap\limits_{s>t} \F_s$ and consider $X\in \bigcup\limits_{r>t} \T_r$.
Then there exists a $t_0 > t$ such that $X \in \T_{t_0}$. 
On the other hand, $Y \in \F_s$ for all $s > t$.
In particular $Y \in \F_{t_0}$.
Hence $\Hom_{\A}(X, Y)=0$ because $(\T_{t_0}, \F_{t_0})$ is a torsion pair.
This implies that $\bigcap\limits_{s>t} \F_s \subseteq \F'$.
In the other direction, let $Y \in \F'$ and consider $X\in \T_{t_0}$ for some $t_0 > t$. 
Then we have that $\Hom_\A (X,Y)=0$ because $\T_{t_0} \subseteq \bigcup\limits_{r>t} \T_r$.
Hence, $\F' \subseteq \F_{t_0}$ for all $t_0 > t$. 
In particular, $\F' \subseteq \bigcap\limits_{s>t} \F_s$.
\end{proof}

Let $\eta$ be a chain of torsion classes in $\A$ and let $M$ be an object in $\A$.
Then $M$ induces two subsets of $[0,1]$ naturally:
$$\J^-_{M} := \{ s : M \in \T_s \}   \text{\qquad and \qquad} \J^+_{M} := \{r : M \not \in \T_r\}. $$
Note that if $M$ is non-zero then both sets defined above are non empty since $0\in\J^-_{M}$ and $1\in\J^+_{M}$

\begin{lemma}\label{lem:Dedekindcuts}
Let $\eta\in \CT(\A)$ be a chain of torsion classes, $M$ be a non-zero object of $\A$ and let $t= \inf \J_{M}^+ = \sup \J_{M}^-$.
Then 
$$\bigcup\limits_{r\in \mathcal{J}_M^+} \T_r \subseteq \T_{t} \subseteq \bigcap\limits_{r\in \mathcal{J}_M^-} \T_r.$$
Moreover either $\T_t = \bigcup\limits_{r\in \mathcal{J}_M^+} \T_r$ or $\T_t = \bigcap\limits_{r\in \mathcal{J}_M^-} \T_r$.
\end{lemma}

\begin{proof}
By hypothesis we have that $\J_{M}^{+}$ is non empty.
Then there exist $r\in \J_{M}^{+}$ such that $M \not \in \T_r$ and $r$ is a upper bound for $\J_{M}^-$.
Hence there exists $t \in [0,1]$ such that $t= \sup \J_{M}^-$. 
In particular we have that $t\geq s$ for every $s\in\J_M^-$.
Therefore $\T_{t} \subseteq \T_s$ for every $s \in \J_M^-$.
We can then conclude that $\T_{t} \subseteq \bigcap\limits_{r\in \mathcal{J}_M^-} \T_r$.
A similar argument shows that $\bigcup\limits_{r\in \mathcal{J}_M^+} \T_r \subseteq \T_{t}$.
The moreover part of the statement follows from the fact that either $t \in \J_M^+$ or $t\in \J_M^-$.
\end{proof}

\begin{remark}\label{rmk:noT}
Note that Lemma \ref{lem:Dedekindcuts} implies that, for every object $M \in \A$, there is no $t \in [0,1]$ such that the torsion class $\T_t$ in $\eta$ verifies
$$\bigcup\limits_{M \not\in \T_s} \T_s \subsetneq \T_t \subsetneq \bigcap\limits_{M \in \T_s} \T_s.$$
However, this is not true for a general torsion class in $\A$, in the sense that there might be a torsion class $\T$ in $\A$ such that 
$$\bigcup\limits_{M \not\in \T_s} \T_s \subsetneq \T \subsetneq \bigcap\limits_{M \in \T_s} \T_s$$
but in that case it does not belong to $\eta$.
Take for instance the chain of torsion classes $\eta$ built in Example~\ref{ex:nowide} and consider the object $M=\rep{1\\2} \in \mod A$ and the torsion class $\add\left\{\rep{1}\oplus \rep{1\\2}\right\}\subseteq \mod A$. 
Then the following holds.
$$\bigcup\limits_{r\in \mathcal{J}_M^+} \T_r =\add\left\{\rep{1} \oplus \rep{2} \oplus \rep{2\\3} \oplus \rep{1\\2} \oplus \rep{1\\2\\3}\right\} \subsetneq \add\left\{\rep{1}\oplus \rep{1\\2}\right\} \subsetneq \add\left\{\rep{1}\right\}=\bigcap\limits_{r\in \mathcal{J}_M^-} \T_r$$
Also, it is important to remark that $\bigcup\limits_{M \not\in \T_s} \T_s $ is strictly contained in $ \bigcap\limits_{M \in \T_s} \T_s$.
\end{remark}

In the following definition we take a chain of torsion classes $\eta$ and for every $t \in [0,1]$ we construct a subcategory $\P_t \subseteq \A$ that will be essential for the rest of the paper.

\begin{definition}\label{def:quasisemistables}
Let $\eta \in \CT(\A)$.
Then we define $\P_\eta:= \{\P_\eta(t) : t \in [0,1]\}$ where, for every $t \in [0,1]$, $\P_\eta(t)$ is a full-subcategory of $\A$ defined as follows.
$$
\P_\eta(t)= \begin{cases}
\bigcap\limits_{s>0} \F_s  & \text{ if $t = 0$}\\
\bigcap\limits_{s<t} \T_s  \cap \bigcap\limits_{s>t} \F_s & \text{ if $t \in (0,1)$}\\
\bigcap\limits_{s<1} \T_s  & \text{ if $t = 1$}
\end{cases}
$$
For every $\eta \in \CT(\A)$ and $t\in [0,1]$, we say that  $\P_\eta(t) \in \P_\eta$ is the category of \textit{$\eta$-quasisemistable} objects of \textit{phase} $t$.
\end{definition}

\begin{figure}
\begin{center}
    \begin{tikzpicture}
    \begin{scope}
        \clip (0,0) ellipse (4 and 2);
        \clip (3,0) ellipse (4 and 3);
        \clip (-3,0) ellipse (4 and 3);
        \filldraw[gray!40] (0,0) ellipse (4 and 2);
    \end{scope}
    \begin{scope}
    \clip (0,0) ellipse (4 and 2);
    \draw[] (3,0) ellipse (4 and 3);
    \draw[] (2.6,0) ellipse (4 and 3);
    \draw[] (2.8,0) ellipse (4 and 3);
    \draw[] (2.9,0) ellipse (4 and 3);
    \draw[] (2.95,0) ellipse (4 and 3);
    \draw[] (2.975,0) ellipse (4 and 3);
    \draw[] (-3,0) ellipse (4 and 3);
    \draw[] (-2.6,0) ellipse (4 and 3);
    \draw[] (-2.8,0) ellipse (4 and 3);
    \draw[] (-2.9,0) ellipse (4 and 3);
    \draw[] (-2.95,0) ellipse (4 and 3);
    \draw[] (-2.975,0) ellipse (4 and 3);
    \draw[thick] (0,0) ellipse (4cm and 2cm);
    \end{scope}
    \node[] at (0, 0)  {$\P_\eta(t)$};
    \node[] at (2.3, 0)  {$\bigcap\limits_{s<t} \T_s$};
    \node[] at (-2.3, 0)  {$\bigcap\limits_{s>t} \F_s$};
    \node[anchor=north west] at (3, 2)  {$\A$};
\end{tikzpicture}
\end{center}
    \caption{The category of quasisemistable objects (in grey).}
    \label{fig:quasisemistables}
\end{figure}

The term \textit{quasisemistables} will become clear in Section \ref{sec:stab}, where we compare the chain of torsion classes with the stability conditions introduced in \cite{Rudakov1997} and \cite{Bridgeland2007}.
Now we prove the main result of this section.

\begin{theorem}\label{thm:algHNfilt}
Let $\A$ be an abelian category and $\eta$ be a chain of torsion classes in $\CT(\A)$.
Then every non-zero object $M$ of $\A$ admits a Harder-Narasimhan filtration.
\end{theorem}

\begin{proof}
Let $M \in \A$, $r_{n} = \inf{\{r \in [0,1] : M \not\in \T_r\}}$ and consider the canonical short exact sequence 
$$0 \longrightarrow t_{n} M \longrightarrow M \longrightarrow M/t_{n} M \longrightarrow 0$$
with respect to the torsion pair $\left( \bigcup\limits_{r>r_n} \T_r ,  \bigcap\limits_{s>r_n} \F_s
\right)$. 

If $r_n=1$, then $M\in \T_r$ for every $r < 1$, we have that $M \in \bigcap\limits_{r<1} \T_r= \P_\eta(1)$. 
In other words the filtration $0 \subset M$ is a Harder-Narasimhan filtration for $M$. 

Otherwise, we have that $M \not \in \bigcup\limits_{r>r_n} \T_r$, implying that $t_{n} M$ is a proper subobject of $M$.
We claim that $M/t_{n} M \in \P_\eta(r_n)$. 
Indeed $M/t_n M \in \bigcap\limits_{r>r_n} \F_r$ by construction. 
On the other hand, $M \in \T_s$ for all $s < r_n$, which implies that $M \in \bigcap\limits_{s<r_n} \T_s$.
Moreover, $\bigcap\limits_{s<r_n} \T_s$ is a torsion class by Proposition \ref{prop:limsup}.
In particular $\bigcap\limits_{s<r_n} \T_s$ is closed under quotients, so $M/ t_n M \in \bigcap\limits_{s<r_n} \T_s$. 
Hence 
$M/ t_n M \in \P_\eta(r_n)=\bigcap\limits_{s<r_n} \T_s \cap \bigcap\limits_{r>r_n} \F_r$ as claimed. 

Set $M_{n-1}= t_n M$ and repeat the process we did for $M_n$.
That is, set 
$$r_{n-1} = \inf\{r \in [0,1] : M_{n-1} \not\in \T_r\}$$
and consider the canonical short exact sequence 
$$0 \longrightarrow t_{n-1} (M_{n-1}) \longrightarrow M_{n-1} \longrightarrow M_{n-1}/t_{n-1} (M_{n-1}) \longrightarrow 0$$
with respect to the torsion pair $\left(\bigcup\limits_{r>r_{n-1}} \T_r, \bigcap\limits_{s>r_{n-1}} \F_s\right)$. 
Again, if $r_{n-1}=1$ the process stops. 
Otherwise one shows that \mbox{$M_{n-1}/t_{n-1} (M_{n-1}) \in \P_\eta(r_{n-1})$} by the arguments above. 
By construction we have that $M_{n-1} \in \bigcup\limits_{r>r_n} \T_r \subseteq \T_{r_n}$.
In particular, this implies that $r_n < r_{n-1}$.

Inductively, we construct a filtration
$$\dots \subset M_{n-2} \subset M_{n-1} \subset M_n$$
of $M$ such that $M_{n-i} / M_{n-(i+1)}\in \P_\eta(r_{n-i})$, where $r_{n-i} > r_{n-j}$ if and only if $i > j$.
Also, we have that $M_{n-1} \in \bigcap\limits_{s>r_n} \T_s \subseteq \bigcap\limits_{s>0} \T_s$. 
By construction, $M_{n-(i+1)}$ is a proper subobject of $M_{n-i}$ for each $i$ and the objects $M_{n-i}$ and $M_{n-j}$ are pairwise non isomorphic if $i$ is different form $j$.

We claim that the filtration constructed above is finite. 
Suppose to the contrary that it is not. 
Then there exists an infinite set $\{r_{n-i} : i \in \mathbb{N}\} \subset [0,1]$ where $r_{n-i}$ is the phase of the $\eta$-quasisemistable object $M_{n-i}/M_{n-(i+1)}$.
Since $\eta$ is weakly-artinian by definition, we have the existence of a $s(M)= r_{n-k} \in \{r_{n-i} : i \in \mathbb{N}\}$ such that $t_{s(M)}M$ is a subobject of $t_{r_{n-i}} M$ for every $i \in \mathbb{N}$.
In particular we have that $M_{n-k}$ is a subobject of $M_{n-(k-1)}$.
However, by construction, $M_{n-(k-1)}$ is the torsion subobject of $M_{n-k}$. 
In particular this implies that the identity map of $M_{n-{k-1}}$ factors through the inclusion of $M_{n-k}$ to $M_{n-(k-1)}$. 
As a consequence we can conclude that $M_{n-k}$ and $M_{n-(k-1)}$ are isomorphic, a contradiction with the properties of the objects of our construction.

Finally, the uniqueness up to isomorphism of this filtration follows from the fact that in this process each $r_i$ is completely determined by the object $M_{i}$, each object $M_i$ is completely determined by $r_{i+1}$ and that torsion subobjects are completely determined up to isomorphism. 
\end{proof}

It is worth noticing that similar filtrations induced by (generalised) tilting objects already exist in the literature.
We recommend the interested reader to see \cite{MTP2019} and the references therein.

\begin{example}\label{ex:torsionpair2}
In Example~\ref{ex:torsionpair} we have introduced a chain of torsion classes $\eta_{(a,b)} \in \CT(\A)$. 
A small calculation shows that $\P_{\eta_{(a,b)}}= \{\P_{\eta_{(a,b)}}(t) : t \in [0,1]\}$, where $\P_{\eta_{(a,b)}}(t)$ is as follows. 
$$
\P_{\eta_{(a,b)}}(t)= \begin{cases}
 \F  & \text{ if $t = a$}\\
 \T & \text{ if $t = b$}\\
 \{0\}  & \text{otherwise}
\end{cases}
$$
In particular, in this case we have that, for every $M \in \A$, the Harder-Narasimhan filtration given by Theorem \ref{thm:algHNfilt} coincides with the canonical short exact sequence of $M$ with respect to the torsion pair $(\T,\F)$.
\end{example}

\begin{example}
Let $\eta_f \in \CT(\mod \mathbb{C}Q)$ be the chain of torsion classes built in Example~\ref{ex:continousP}.
In this case, $\P_\eta(0)$ corresponds to the additive category generated by all preprojective objects and $\mbox{add}\{M(\lambda, n): f_z(\lambda) = 0\}$, the category $\P_\eta(1)$ is the additive category generated by all preinjective objects and $\mbox{add}\{M(\lambda, n): f_z(\lambda) = 1\}$, and for every $t\in(0,1)$ the category $\P_{\eta_{f}}(t) = \mbox{add}\{M(\lambda, n): f_z(\lambda) = r \text{ and } n \in \mathbb{N}\}\cup\{0\}$. 
The Harder-Narasimhan for every indecomposable $M$ is of the form $0 \subset M$.
\end{example}


\section{A preorder induced by a chain of torsion classes}\label{sec:ordering}

Let $W([0,1])$ be the set of all the words of finite length that can be formed using elements in $[0,1]$.
This set has a natural total order given by the lexicographic order. 
In the previous section we have shown that every chain of torsion classes $\eta \in \CT(\A)$ induces a Harder-Narasimhan filtration in every non-zero object $M$ of $\A$.
In this section we use this result to define a function $\varphi_\eta$ from the non-zero objects of $\A$ to $W([0,1])$ inducing a preorder among the objects of $\A$ and we study some of its properties.

\begin{definition}\label{def:phieta}
Let $\A$ be an abelian category and let $\eta$ be a chain of torsion classes in $\CT(\A)$.
For every non-zero object $M$ in $\A$ with Harder-Narasimhan filtration 
$$M_0 \subset M_1 \subset \dots \subset M_n$$
where, for every $1\leq k\leq n$, $M_k/M_{k-1} \in \P_{\eta}(r_k)$ we associate to $M$ the word
$$\varphi_\eta (M) := r_n r_{n-1} \dots r_1$$
in $W([0,1])$.
We say that $M \leq_\eta N$ if $\varphi_\eta(M) \leq \varphi_\eta(N)$.
\end{definition}

It follows from Theorem~\ref{thm:algHNfilt} that the preorder $\leq_\eta$ is constant in isomorphism classes of objects. 
However, $\varphi_\eta$ is not a partial order in the set of isomorphism classes of objects in $\A$ since $\varphi_\eta$ is not antisymmetric as it follows from the next proposition.

\begin{theorem}\label{thm:quasistables}
Let $\A$ be an abelian category, $\eta$ be a chain of torsion classes in $\CT(\A)$ and let $t,t' \in [0,1]$.
Then:
\begin{enumerate}
\item $\varphi_\eta (M) = t$ if and only if $M \in \P_{\eta}(t)$;
\item  if $M \in \P_{\eta}(t)$ then $\varphi_\eta (M) \leq \varphi_\eta (N)$ for all quotient $N$ of $M$;
\item if $M \in \P_{\eta}(t)$ then $\varphi_\eta (L) \leq \varphi_\eta (M)$ for all subobject $L$ of $M$;
\item If $t' < t$, $M \in \P_{\eta}(t)$ and if $M' \in \P_{t'}$ then $\emph{Hom}_{\A}(M,M')=0$;
\item $\P_{\eta}(t)$ is closed under extensions for all $t \in [0,1]$.
\end{enumerate}
\end{theorem}

\begin{proof}
\textit{1.}  Let $M \in \P_{\eta}(t)$. 
Then $0 \subset M$ is a Harder-Narasimhan filtration for $M$.
Therefore, $0 \subset M$ is the Harder-Narasimhan filtration for $M$ because it is unique up to isomorphism.
Hence $\varphi_\eta(M)=t$ by Definition \ref{def:phieta}.

In the other direction, let $M$ such that $\varphi_\eta (M) = t$. 
Then, by definition of $\varphi_\eta$, we have that $0 \subset M$ is the Harder-Narasimhan filtration of $M$.
Moreover, Theorem \ref{thm:algHNfilt} also implies that $M / 0 \cong M \in \P_{\eta}(t)$, as claimed.

\textit{2. and 3.} We only show 2. since the proof of 3. is similar.
Let $M$ be an object of $\P_{\eta}(t)$. 
Then, $\varphi_\eta (M) = t $ by \textit{1}.
We also have that $M \in \P_{\eta}(t) =\bigcap\limits_{r<t} \T_{r}  \cap \bigcap\limits_{s>t} \F_s$.
In particular $M \in \bigcap\limits_{r<t} \T_{r}$.
Therefore $N \in \bigcap\limits_{r<t} \T_{r}$ since $\bigcap\limits_{r<t} \T_{r}$ is a torsion class.

Now, from the proof of Theorem \ref{thm:algHNfilt} we have that the first letter $r'_{n'}$ of 
$$\varphi_\eta (N) = r'_{n'} r'_{n'-1} \dots r'_1$$
is $r'_{n'}= \inf \{ r : N \not \in \T_r\}$.
The fact that $N \in \bigcap\limits_{r<t} \T_{r}$ implies that $t \leq r'_{n'}$.
Hence $\varphi_\eta (M) \leq \varphi_\eta (N)$ for every non-zero quotient $N$ of $M$.

\textit{4.} Let $t, t' \in [0,1]$ such that $t' < t$, $M \in \P_{\eta}(t)$, $M'\in \P_{\eta}(t')$ and $f: M \to M'$. 
Then we have that the image $\im f$ of $f$ is a quotient of $M$ and a subobject of $M'$. 
Then \textit{2.} implies that $\varphi_\eta (\im f) \geq t$ while \textit{3.} implies that $\varphi_\eta (\im f) \leq t'$. 
Hence we have that $\varphi_\eta (\im f)$ is not well-defined. 
Since $\varphi_\eta$ is well-defined for all non-zero objects, we can conclude that $f$ is the zero morphism.
Therefore $\Hom_\A (M,M')=0$, as claimed.

\textit{5.} Recall that $\P_{\eta}(t)$ is defined as $\P_t :=\bigcap\limits_{s<t} \T_s  \cap \bigcap\limits_{s>t} \F_s$.
Then, $\P_{\eta}(t)$ is the intersection of a torsion class and a torsion free class. 
Therefore $\P_{\eta}(t)$ is always closed under extensions.
\end{proof}

\begin{remark}
In the fifth point of the previous result we showed that $\P_{\eta}(t)$ is closed under extensions for every $t\in [0,1]$. 
In fact, it follows from \cite[Theorem 3.2]{Tattar2019} that $\P_{\eta}(t)$ is a full quasi-abelian subcategory of $\A$ for every $t \in [0,1]$.
The following example shows that, in general, $\P_{\eta}(t)$ is not a wide subcategory of $\A$.
\end{remark}

\begin{example}\label{ex:nowide2}
We continue with the chain of torsion classes of Example~\ref{ex:nowide}.
This chain of torsion classes $\eta$ is such that for $t\in [0,1]$, the torsion class $\T_t$ is of the following form.  

$$
\T_t= \left\{\begin{array}{ll}
\mod A & \text{ if $t=0$}\\
add\left\{\rep{1} \oplus \rep{2} \oplus \rep{2\\3} \oplus \rep{1\\2} \oplus \rep{1\\2\\3}\right\} & \text{ if $t \in (0,1/3]$}\\
add\{\rep{1}\} & \text{ if $ t \in (1/3, 2/3]$}\\
\{0\} & \text{ if $t\in (2/3,1]$}
\end{array}
\right.
$$

For this chain of torsion classes we have that $\P_{\eta}(t) =\{0\}$ for every $t$ different to $0$, $1/3$ or $2/3$, where $\P_{\eta}(0)=add\{\rep{3}\}$, $\P_{\eta}(1/3)=add\left\{\rep{2} \oplus \rep{2\\3} \oplus \rep{1\\2} \oplus \rep{1\\2\\3}\right\}$ and $\P_{\eta}(2/3)=add\{\rep{1}\}$, respectively.
Now take the inclusion $f: \rep{2\\3} \to \rep{1\\2\\3}$. 
Here, we have that $\rep{2\\3}$ and $\rep{1\\2\\3}$ belong to $\P_{\eta}(1/3)$, however $\coker f \cong \rep{1}$ does not belong to $\P_{\eta}(1/3)$. 
So $\P_{\eta}(1/3)$ is not closed under cokernels.
Similarly, we can take the epimorphism $g : \rep{1\\2\\3} \to \rep{1\\2}$.
Here $\rep{1\\2\\3}$ and $\rep{1\\2}$ belong to $\P_{\eta}(1/3)$, but $\ker g \cong \rep{3}$ does not. 
Hence, is not $\P_{\eta}(1/3)$ closed under kernels either.
\end{example}

Take a chain of torsion classes $\eta \in \CT(\A)$ and consider a non-zero object $M \in \A$. 
In some sense, Theorem \ref{thm:quasistables}.2 says that it is expected that $\varphi_{\eta} (M) \leq \varphi_\eta (N)$ whenever $N$ is a quotient of $M$. 
Therefore it is of interest to know if there exists a quotient $N'$ of $M$ such that $\varphi_\eta (N') \leq \varphi_{\eta} (N)$ for every quotient $N$ of $M$. 
The following proposition shows that such objects always exists.

\begin{proposition}\label{prop:MDQ}
Let $\eta\in \CT(\A)$ be a chain of torsion classes and let $M \in \A$ with Harder-Narasimhan filtration $M_0 \subset M_1 \subset \dots \subset M_n$.
Define $M_\eta^- := M_n/M_{n-1}$ and $M_\eta^+:= M_1$.

Then $\varphi_\eta (N) \geq \varphi_\eta (M_\eta^-)$ for all quotient $N$ of $M$.
Moreover, if the equality holds there exists an epimorphism $p: M_\eta^- \to N$ making the following diagram commutative.
\begin{center}
	\begin{tikzcd}
   	 M \ar[r] \ar[d, "id_M"]& M_\eta^- \ar[r]\ar[d, "p"] & 0 \\
     M \ar[r]& N \ar[r] & 0 
	\end{tikzcd}
\end{center}

Dually, $\varphi_\eta (L) \leq \varphi_\eta (M_\eta^+)$ for all subobject $N$ of $M$.
Moreover, if the equality holds there exists an monomorphism $i: L \to M_\eta^+$ making the following diagram commutative.
\begin{center}
	\begin{tikzcd}
   	 0 \ar[r] &	L \ar[r]\ar[d, "i"] & M \ar[d, "id_M"] \\
     0 \ar[r] &	M_\eta^+ \ar[r] & M 
	\end{tikzcd}
\end{center}
\end{proposition}

\begin{proof}
We prove the statement for the quotients, being similar the proof for subobjects.
If $N$ be a quotient of $M$, then $N \in \T_s$ for all torsion class $\T_s$ such that $M \in \T_s$.
Thus $t = \inf \{ r \in [0,1] : M \not \in \T_r \} \leq \inf \{r' \in [0,1] : N \in \T_{r'}\}$.
Hence $\varphi_\eta (M_n / M_{n-1}) = t \geq \inf \{ r \in [0,1] : M \not \in T_r \} \geq \varphi_\eta (N)$, as claimed.

Now we prove the moreover part of the statement. 
If $\varphi_\eta (N) = \varphi_\eta (M_n/M_{n-1}) = t$, we have that $N \in \P_{\eta}(t)$ by Theorem \ref{thm:quasistables}.1.
In particular $N \in \bigcap\limits_{s>t} \F_s$. 
But, $M_n / M_{n-1}$ is the torsion free quotient of $M$ with respect to the torsion pair $$\left( \bigcup\limits_{s>t} \T_s , \bigcap\limits_{s>t} \F_s \right).$$
Therefore, the properties of the canonical short exact sequence imply the existence of a surjective map $p: M_n/M_{n-1} \to N$ making diagram 
\begin{center}
	\begin{tikzcd}
   	 M \ar[r] \ar[d, "id_M"]& M_n/M_{n-1} \ar[r]\ar[d, "p"] & 0 \\
     M \ar[r]& N \ar[r] & 0 
	\end{tikzcd}
\end{center}
commutative.
This finishes the proof.
\end{proof}

\begin{remark}
Following the terminology of stability conditions, given a chain of torsion classes $\eta \in \CT(\A)$ and a non-zero object $M$ of $\A$ we say that $M^-_\eta$ is the \textit{maximally destabilizing quotient} of $M$ with respect to $\eta$.
Dually, we say that $M^+_\eta$ is the \textit{maximally destabilizing subobject} of $M$ with respect to $\eta$.
\end{remark}

Recall that an object $M$ in a subcategory $\mathcal{B}$ of $\A$ is said to be \textit{relatively simple} if the only subobjects of $M$ that belong to $\mathcal{B}$ are $0$ and $M$ itself. 
Dually, an object $M$ in a subcategory $\mathcal{B}$ of $\A$ is said to be \textit{relatively cosimple} if the only quotients of $M$ that belong to $\mathcal{B}$ are $0$ and $M$ itself. 
Also, we say that an object $M$ in $\A$ is a \textit{brick} if its endomorphism algebra End$_\A (M)$ is a division ring.

In the following proposition we further study the morphisms between objects in $\P_{\eta}(t)$ for any $t\in [0,1]$.

\begin{proposition}\label{prop:Endstables}
Let $\eta \in \CT(\A)$, $t \in [0,1]$ and consider a non-zero morphism $f: M \to N$, where $M,N \in \P_{\eta}(t)$. 
Then following statements hold:
\begin{enumerate}
\item The object $\im f$ belongs to $\P_{\eta}(t)$.
\item If $N$ is relatively simple in $\P_{\eta}(t)$ then $f$ is an epimorphism;
\item If $M$ is relatively cosimple in $\P_{\eta}(t)$ then $f$ is a monomorphism;
\item If $M$ is relatively cosimple object and $N$ is relatively simple in $\P_{\eta}(t)$ then $f$ is an isomorphism;
\item If $M$ is a relatively simple and cosimple object in $\P_{\eta}(t)$, then $M$ is a brick.
\end{enumerate}
\end{proposition}

\begin{proof}
\textit{1.}
Let $M$ and $N$ be two objects of $\P_{\eta}(t)$ and $f: M \to N$ be a non-zero morphism.
Then $\im f$ is a quotient of $M$ and a subobject of $N$.
Therefore, Theorem \ref{thm:quasistables}.3 and Theorem \ref{thm:quasistables}.2 imply that $t \leq \varphi_\eta (\im f) \leq t$.
So, we can conclude that $\varphi_\eta (\im f)=t$.
Hence we have that $ \im f \in \P_{\eta}(t)$ by Theorem \ref{thm:quasistables}.1.

\textit{2.}
We have that $\im f \in \P_{\eta}(t)$ by 1.
Also, we have by hypothesis that $N$ is relatively simple, so $\im f$ is either $0$ or $N$. 
Since $f$ is non-zero, we have that $\im f \cong N$.
Hence $f$ is an epimorphism.

\textit{3.}
We have that $\im f \in \P_{\eta}(t)$ by 1.
Also, we have by hypothesis that $M$ is relatively cosimple, so $\im f$ is either $0$ or $M$. 
Since $f$ is non-zero, we have that $\im f \cong M$.
Hence $f$ is a monomorphism.

\textit{4.}
Follow directly from 2. and 3.

\textit{5.}
Is a particular case of 4.
\end{proof}


\section{Chains of torsion classes and stability conditions}\label{sec:stab}


Harder-Narasimhan filtrations first arose in algebraic geometry as a way to understand all vector bundles of the projective line in terms of stable vector bundles. 
The concept of stability was then adapted to quiver representations in \cite{King1994} and \cite{Schofield1991}.
Later, stability conditions for abelian categories were studied in \cite{Rudakov1997}.

This section is devoted to compare important concepts of stability conditions, such as Harder-Narasimhan filtrations or (semi)stable objects, with the results we obtained in Section \ref{sec:ordering}.
We start the present section recalling the definition of a stability function, which is equivalent to the stability conditions defined in \cite{Rudakov1997} (see \cite[Section 2]{BSTpath}).
In this section we consider noetherian abelian categories in order to avoid technical difficulties.

\begin{definition}
Let $\A$ be a noetherian abelian category and let $\phi$ be a function $\phi :  Obj^*(\mathcal{A}) \to [0,1]$ sending the non-zero objects $Obj^*(\A)$ of $\A$ to $[0,1]$ which is constant on isomorphism classes. 
We say that $\phi$ is a stability function if for each short exact sequence  $0 \to L \to M \to N \to 0$ of non-zero objects in $\mathcal{A}$ one has the so-called \textit{see-saw} (or \textit{teeter-totter}) property:
$$
\begin{array}{ll}
\text{either} & \phi(L) < \phi(M) < \phi(N), \\
\text{or} & \phi(L) > \phi(M) > \phi(N),\\
\text{or} & \phi(L) = \phi(M) = \phi(N).
\end{array}
$$
\end{definition}
We recall the definition of $\phi$-(semi)stable objects.
\begin{definition}
Let $\phi:Obj^*(\mathcal{A})\rightarrow [0,1]$ be a stability function on $\mathcal{A}$. 
A non-zero object $M$ in $\mathcal{A}$ is said to be \textit{$\phi$-stable} (or \textit{$\phi$-semistable}) if every nontrivial subobject $L\subset M$ satisfies $\phi(L)< \phi(M)$ ($\phi(L)\leq \phi(M)$, respectively).
\end{definition}

As for chains of torsion classes, it is not true that every stability function induces Harder-Narasimhan filtrations.
In order to be able give proof of the main results in \cite{Rudakov1997}, Rudakov introduced the notions of quasi-noetherian and weakly-artinian abelian categories.
Here, instead of imposing these conditions on to the categories we impose them to the stability conditions.

\begin{definition}\label{def:StabA}
Let $\phi: Obj \A \to [0,1]$ be a stability function. 
We say that $\phi$ is \textit{quasi-noetherian} if all chains of proper subobjects $X_1 \subset X_2 \subset \dots \subset X$ of $X$ such that $\phi(X_i) \leq \phi(X_{i+1})$ for every $i\in \mathbb{N}$ stabilises. 
Similarly, we say that $\phi$ is \textit{weakly-artinian} if all chains of proper subobjects $X\supset X_1 \supset X_2 \supset \dots $ of $X$ such that $\phi(X_i) \leq \phi(X_{i+1})$ for every $i\in \mathbb{N}$ stabilises. 

We denote by $\Stab\A$ the set of all quasi-noetherian and weakly-artinian stability conditions. 
\end{definition}

One of the main objectives of \cite{BSTpath} was to study the relation between stability functions and torsion pairs.
It was proven that given a stability function $\phi$ one can define a torsion class for every $t \in [0,1]$.
We include the proof for the convenience of the reader.

\begin{proposition}\cite[Proposition 2.15, Proposition 2.17]{BSTpath}\label{prop:stabtorsion}
Let $\phi:Obj^*(\mathcal{A})\rightarrow [0,1]$ be a stability function in $\Stab(\A)$ and let $p \in [0,1]$. Then:
\begin{itemize}
\item $\T_{\geq p} := \{ M \in \A : \phi(N) \geq p \text{ for all quotient $N$ of $M$} \} \cup \{ 0 \}$ is a torsion class;
\item $\T_{>p} := \{ M \in \A : \phi(N) > p \text{ for all quotient $N$ of $M$} \} \cup \{ 0 \}$ is a torsion class;
\item $\F_{\leq p} := \{ M \in \A : \phi(L) \leq p \text{ for all subobject $L$ of $M$} \} \cup \{ 0 \}$ is a torsion free class;
\item $\F_{< p} := \{ M \in \A : \phi(L) < p \text{ for all subobject $L$ of $M$} \} \cup \{ 0 \}$ is a torsion free class.
\end{itemize}
Moreover, $(\T_{\geq p}, \F_{<p})$ and $(\T_{> p}, \F_{\leq p})$ are torsion pairs in $\A$.
\end{proposition}

\begin{proof}
We only prove that $\T_{\geq p}$ is a torsion class, since the proofs of the other bullet points are similar.
Since $\A$ is noetherian, it is enough to show that $\T_{\geq p}$ is closed under quotients and extensions.

We first show that $\T_{\geq p}$ is closed under quotients. 
Let $M \in \T_{\geq p}$ and $N$ be a quotient of $M$.
Then every quotient $N'$ of $N$ is also a quotient of $M$. 
So, $\phi(N') \geq p$ because $M \in \T_{\geq p}$. 
This shows that $\T_{\geq p}$ is closed under quotients. 

Now we show that $\T_{\geq p}$ is closed under extensions. 
Consider the short exact sequence 
$$0\rightarrow L\overset{f}\rightarrow M\overset{g}\rightarrow N\rightarrow 0$$ 
such that $L, N \in \T_{\geq p}$ and let $M'$ be a quotient of $M$.
Then we can build the following commutative diagram.
$$\xymatrix{
 0\ar[r]& L\ar[r]^f\ar[d] &M\ar[r]^g\ar[d]^{p} &N\ar[r]\ar[d] &0 \\
 0\ar[r]& \im (p f)\ar[r]^{f'}\ar[d] & M'\ar[r]^{g'}\ar[d] &\coker (f')\ar[r]\ar[d] &0 \\
  & 0 & 0 & 0 & }$$
If $\im (pf) = 0$, we have that $M'$ is isomorphic to a quotient of $N$, implying that $\phi(M')\geq p$ be cause $N \in \T_{\geq p}$.
Similarly, if $\coker (f') = 0$ we have that $\phi(M') \geq p$.

Otherwise, $\im (pf)$ and $\coker (f')$ are non-zero objects in $\A$ and then $\phi(\im (pf))$ and $\phi(\coker (f'))$ are well-defined.
Moreover, $\phi(\im (pf)) \geq p$ and $\phi(\coker (f')) \geq p$ because $\im (pf)$ and $\coker (f')$ are quotient of $L$ and $M$, respectively. 
Then, we have that $\phi(M') \geq p$ because $\phi$ is a stability function. 
Hence $\T_{\geq p}$ is a torsion class.

For the moreover part of the statement, we only prove that $(\T_{\geq p}, \F_{< p})$ is a torsion pair in $\A$.
The same proof also works for $(\T_{>p}, \F_{\leq p})$.
Consider let $f: M \to N$, where $M \in \T_{\geq p}$ and $ N \in \F_{< p}$.
Then we have that $\im f$ is a quotient of $M$ and a subobject of $N$. 
Then $\im f \cong 0$ because $\phi$ is a stability function. 
That is $\Hom_{\A} (M , N)=0$.

Now, consider $X$ such that $\Hom_\A (X,Y)=0$ for every $Y \in \F_{<p}$.
By \cite[Proposition 1.9]{Rudakov1997} there exists a quotient $M'$ of $X$ such that $M'$ is $\phi$-semistable and $\phi(M')\leq \phi(N)$ for every quotient $N$ of $X$.
Since $\Hom_\A (X,Y)=0$ for every $Y \in \F_{<p}$, we conclude that $p \leq \phi(M')$.
Hence $\phi(N) \geq p$ for every quotient $N$ of $X$.
Then $X \in \T_{\geq p}$ by definition. 
One shows that every $Y$ such that $\Hom_\A (X,Y)=0$ for every $X \in \T_{\geq p}$ belong to $\F_{> p}$ similarly.
\end{proof}

A direct consequence of Proposition~\ref{prop:stabtorsion} is the following lemma which appeared already in \cite{Baumann2013} and \cite{BSTpath} as a side commentary. 
We choose to state it clearly as a lemma since it is a key result to compare Harder-Narasimhan filtrations coming from a stability function and the chain of torsion classes it induces. 

\begin{lemma}\label{lem:stab-chain}
Let $\phi:Obj^*(\mathcal{A})\rightarrow [0,1]$ be a stability function in $\Stab(\A)$.
Then $\phi$ induces two chains of torsion classes $\eta_\phi$ and $\overline{\eta}_\phi$ in $\CT(\A)$ as follows.
$$\eta_\phi:= \{\T_s:=\T_{> s} : s\in [0,1]\}$$
$$\overline{\eta}_\phi:= \{\overline{\T}_s:=\T_{\geq s} : s\in [0,1]\}$$
\end{lemma}

\begin{proof}
We only prove that ${\phi}_\eta$ is a chain of torsion classes in $\CT(\A)$ since the proof for $\overline{\phi}_\eta$ is very similar.
It was shown in Proposition \ref{prop:stabtorsion} that $\T_s = \T_{> s}$ is a torsion class for every $s \in [0,1]$.
Suppose that $r \leq s$ and let $M \in \T_s$.
Then $\phi(N) > s$ for all quotient $N$ of $M$ by definition of $\T_s$.
In particular $\phi(N) > s \geq r$ for all quotient $N$ of $M$.
Therefore $M \in \T_r$.
Hence we have that $\eta_\phi$ is a chain of torsion classes. 

It is left to show that $\eta_\phi \in \CT(\A)$.
By hypothesis $\A$ is noetherian, then it follows immediately that $\eta_\phi$ is quasi-noetherian. 
Then, it is enough to show that $\eta_\phi$ is weakly-artinian. 
Suppose that $\eta_\phi$ is not weakly-artinian. 
Then there exists a subset $S$ of $[0,1]$ having an infinite subset $\{r_i \in S \colon r_i < r_{i+1}, i\in \mathbb{N}\}$ inducing a chain
\begin{align}\label{eq:descending chain}
&&\dots \subset t_{r_3}M \subset t_{r_2}M \subset t_{r_1}M \subset M
\end{align}
where $t_{r_{i+1}}M$ is a proper subobject of $t_{r_i}M$.
We claim that $\phi(t_{r_{i+1}}M)> \phi(t_{r_{i}}M)$. 
It follows from the definition of torsion pairs that $t_{r_{i+1}}M$ is the torsion subobject of $t_{r_{i}}M$ with respect of $\T_{>} r_{i+1}$. 
Then the canonical short exact sequence of $t_{r_{i+1}}M$ with respect to $\left(\T_{> r_{i+1}}, \F_{\leq r_{i+1}}\right)$ is the following.
$$0 \to t_{r_{i+1}}M \to t_{r_{i}}M \to \frac{t_{r_{i}}M}{t_{r_{i+1}}M} \to 0$$
By definition we have that $t_{r_{i+1}}M \in \T_{> r_{i+1}}$ and $\frac{t_{r_{i}}M}{t_{r_{i+1}}M} \in \F_{\leq r_{i+1}}$. 
Then, by construction of the torsion pair we know that $\phi\left(t_{r_{i+1}}M\right) > r_{i+1} \geq \phi\left(\frac{t_{r_{i}}M}{t_{r_{i+1}}M}\right)$.
Since $\phi$ is a stability function, it verifies the see-saw property. 
Hence we can conclude that $\phi\left(t_{r_{i+1}}M\right)  > \phi\left(t_{r_{i}}M\right)$ as claimed.
By hypothesis, $\phi$ is a weakly-artinian stability function, so the chain (\ref{eq:descending chain}) stabilises contradicting our assumption. 
Hence $\eta_\phi \in \CT(\A)$.
\end{proof}

\begin{remark}\label{rmk:notallchains}
It is important to note that not every chain of torsion classes $\eta \in \CT(\A)$ is induced by a stability function $\phi \in \Stab(\A)$. 
Indeed, it has been shown in \cite[Proposition~2.20]{BSTpath} that for every stability function $\phi\in\Stab(\A)$ and every $t \in [0,1]$ the category $\P_{\eta_\phi}(t)$ is a wide subcategory of $\A$. 
However there are examples of chains of torsion classes, such as the one appearing in Example~\ref{ex:nowide}, having categories of quasisemistable objects that are not wide.
\end{remark}

The key result that allow us to compare chains of torsion classes and stability condition is the following.

\begin{theorem}\label{thm:semistabchain}
Let $\phi:Obj^*(\mathcal{A})\rightarrow [0,1]$ be a stability function in $\Stab(\phi)$, $\eta_\phi$ be the chain of torsion classes induced by $\phi$ and $t \in [0,1]$.
Then 
$$\P_{\eta_\phi}(t) = \P_{\overline{\eta}_\phi}(t) = \{M \in \A: M \text{ is $\phi$-semistable and } \phi(M)=t\} \cup \{0\}.$$
\end{theorem}

\begin{proof}
By definition, we have that $\P_{\eta_\phi}(t)=\bigcap\limits_{s<t} \T_s  \cap \bigcap\limits_{s>t} \F_s$.
Then Lemma \ref{lem:Dedekindcuts} implies that $\bigcap\limits_{s<t} \T_s= \T_{\geq t}$ and $\bigcap\limits_{s>t} \F_s = \F_{\leq t}$.
Is clear that $\phi(M)=t$ for all $M \in \bigcap\limits_{s<t} \T_s  \cap \bigcap\limits_{s>t} \F_s$.
Moreover, if $L$ is a subobject of $M$ then $\phi(L) \leq t$ by Theorem~\ref{thm:quasistables}.3. 
Therefore $M$ is a $\phi$-semistable object of phase $t$.

In the other direction, suppose that $M$ is a $\phi$-semistable object of phase $t$.
Then we have that $\phi(L) \leq \phi(M)=t$ for every subobject $L$ of $M$ because $M$ is $\phi$-semistable, implying that $M \in \F_{\leq t}$. 
Dually, $t = \phi(M) \leq \phi(N)$ because $M$ is $\phi$-semistable, so $M \in \T_{\geq t}$.
Therefore $M \in \T_{\geq t} \cap \F_{\leq t} = \P_{\eta_\phi}(t)$, as claimed.

To finish the proof it is enough to realise that one can replace every occurrence of $\P_{{\eta}_\phi}$ by $\P_{\overline{\eta}_\phi}$.
\end{proof}

As a consequence of Theorem \ref{thm:algHNfilt} and Theorem \ref{thm:semistabchain} we recover important results on stability conditions.

\begin{corollary}\cite[Theorem 2]{Rudakov1997}\label{cor:stabHNfilt}
Let $\phi:Obj^*(\mathcal{A})\rightarrow [0,1]$ be a stability function in $\Stab(\A)$.
Then every object $M\in \A$ admits a Harder-Narasimhan filtration.
\end{corollary}

\begin{proof}
Let $\phi: \A \to [0,1]$ be a stability function in $\Stab(\A)$.
Then Lemma~\ref{lem:stab-chain} says that $\phi$ induces two chains of torsion classes $\eta_\phi$ and $\overline{\eta}_\phi$ in $\CT(\A)$.
Then Theorem~\ref{thm:algHNfilt} implies that for every non-zero object $M$ of $\A$ there exists a filtration 
$$M_0 \subset M_1 \subset \dots \subset M_n$$
of $M$ such that:
\begin{enumerate}
\item $0 = M_0$ and $M_n=M$;
\item $M_k/M_{k-1} \in \P_{\eta_\phi}(r_k)$ for some $r_k \in [0,1]$ for all $1\leq k \leq n$;
\item $r_1 > r_2 > \dots > r_n$;
\end{enumerate}
which is unique up to isomorphism. 
It follows from Theorem~\ref{thm:semistabchain} that $M_k/M_{k-1}$ is $\phi$-semistable for each $k$.
It also implies that $\phi(M_1) > \dots > \phi(M_n/M_{n-1})$.
\end{proof}

\begin{corollary}\cite[Theorem 1]{Rudakov1997}
Let $\phi: Obj^*(\A) \to \I$ and a stability function in $\Stab(\A)$, $t \in \I$ and consider a non-zero morphism $f: M \to N$, where $M,N$ are $\phi$-semistable and $\phi(M)=\phi(N)=t$. 
Then the following hold:
\begin{enumerate}
\item $\im f$ is a $\phi$-semistable object and $\phi(\im f)=t$;
\item if $N$ is $\phi$-stable then $f$ is an epimorphism;
\item if $M$ is $\phi$-stable $f$ is a monomorphism;
\item if $M$ and $N$ are $\phi$-stable then $f$ is an isomorphism .
\item every $\phi$-stable object is a brick.
\end{enumerate}
\end{corollary}

\begin{proof}
It follows from Theorem \ref{thm:semistabchain} that 
$$\P_{\eta_\phi}(t) = \{M \in \A: M \text{ is $\phi$-semistable and } \phi(M)=t\} \cup \{0\}.$$
Therefore, an object $M \in \P_t$ is relatively simple if and only if $\phi(L) < \phi(M)$ for every subobject $L$ of $M$. 
This implies that an object $M \in \P_{\eta_\phi}(t)$ is relatively simple if and only if $M$ is $\phi$-stable.
Then the result follows directly from Lemma~\ref{lem:stab-chain} and Proposition \ref{prop:Endstables}.
\end{proof}


\section{Slicings in abelian categories}\label{sec:slicing}

The concept of slicing in a triangulated category $\mathcal{D}$ was introduced in \cite{Bridgeland2007}.
We adapt this notion to abelian categories and we show that every chain of torsion classes $\eta \in \CT(\A)$ induces a slicing of $\A$ and, moreover, that every slicing in $\A$ arises in this way.

\begin{definition}\label{def:slicing}
A \textit{slicing} $\P$ of the abelian category $\A$ consists of full additive subcategories $\P:=\{\P(r) \subseteq \A :r \in [0,1]\}$ satisfying the following axioms:
\begin{enumerate}
\item if $M \in \P(r)$ and $N \in \P(s)$ with $r > s$, then $\Hom_{\A}(M, N)=0$;

\item for each non-zero object $M$ of $\A$ there exists a filtration 
$$0=M_0 \subset M_1 \subset \dots \subset M_n=M$$ 
of $M$ and a set 
$$\{ r_i : 1 \leq i \leq n \text{ and } r_i > r_j  \text{ if $i < j$} \}\subset [0,1]$$
such that $M_i/ M_{i-1} \in \P(r_i)$ is unique up to isomorphism.
\end{enumerate}
We denote by $\P(\A)$ the set of all the slicings on an abelian category $\A$.
\end{definition}

Given a subcategory $\mathcal{B}$ of $\A$, we denote by $\mbox{Filt}(\mathcal{B})$ the subcategory of all objects in $\A$ admitting a filtration whose subquotients belong to $\mathcal{B}$.
More precisely, $M \in \mbox{Filt}(\mathcal{B})$ if and only if there exists a chain of subobjects 
$$0=M_0\subset M_1 \subset \dots \subset M_{t-1} \subset M_t = M$$
of $M$ with $M_k / M_{k-1} \in \mathcal{B}$ for all $1\leq k \leq t$.
The following is an easy consequence of the definition of slicing.

\begin{lemma}\label{lem:quotient}
Let $\P$ be a slicing and $s \in [0,1]$. 
If $X \in \P(s)$ then every non-zero quotient $Y$ of $X$ is in $\Filt \left( \bigcup_{t \geq s} \P(t)\right)$.
\end{lemma}

\begin{proof}
By hypothesis $Y$ is a non-zero object of $\A$. 
Then by Definition~\ref{def:slicing}.2 there exists a filtration 
$$0=Y_0\subset Y_1 \subset \dots \subset Y_{t-1} \subset Y_t = Y$$
such that $Y_i/Y_{i-1} \in \P(r_i)$ with $r_t < r_{t-1} < \dots < r_1$.
Since $Y_t/Y_{t-1}$ is a quotient of $Y$ and $Y$ is a quotient of $X$ we can conclude that $\Hom_\A(X,  Y_t / Y_{t-1})\neq 0$.
Then the contrapositive of Definition~\ref{def:slicing}.1 implies that $s \leq r_t$. 
Hence $Y$ is in $\Filt \left( \bigcup_{t \geq s} \P(t) \right)$.
\end{proof}

\begin{proposition}\label{prop:torsasfilt}
Let $\P$ be a slicing of $\A$.
Then for every $s \in (0,1)$ there are torsion pairs $(\overline{\T}_s, \F_s)$ and $(\T_s, \overline{\F}_s)$ where
$$\overline{\T}_s = \Filt \left( \bigcup_{t \geq s} \P(t) \right),
\T_s = \Filt \left( \bigcup_{t > s} \P(t) \right)
$$
$$
\overline{\F}_s = \Filt \left( \bigcup_{r \leq s} \P_{\eta}(r) \right),
\F_s = \Filt \left( \bigcup_{r < s} \P(r)\right).$$
\end{proposition}

\begin{proof}
We only prove that $(\T_s, \overline{\F}_s)$ is a torsion pair since the proof of the other case is identical. 
Let $\P$ be a slicing of $\A$ and $s\in(0,1)$.

We first show that $\T_s$ is closed under extensions. 
Let $0 \to L \to M \to N \to 0$ be a short exact sequence with $L,N \in \T_s$. 
Then, by assumption, $L$ and $N$ admit a filtration as follows 
$$0=L_0\subset L_1 \subset \dots \subset L_{l-1} \subset L_l = L$$
$$0=N_0\subset N_1 \subset \dots \subset N_{t-1} \subset N_t = N$$
where $L_i/L_{i-1}$ and $N_j/N_{j-1}$ belong to $\bigcup_{t \geq s} \P(t)$ for every $1\leq i \leq l$ and $1 \leq j \leq n$, respectively.
Since $N$ is a quotient of $M$, the filtration of $N$ given above lifts to a filtration 
$$L=M_{l+1}\subset M_{l+2} \subset \dots \subset M_{l+t-1} \subset M_{l+t} = M$$
of $M$ where $M_{l+j}/M_{l+j-1} \cong N_j/N_{j-1} \in \bigcup_{t \geq s} \P(t)$.
Then using the filtration of $L$ we obtain a filtration 
$$0=M_0\subset M_1 \subset \dots \subset M_{l-1} \subset M_l \subset M_{l+1} \subset \dots \subset M_{l+t-1} \subset M_{l+t} = M$$
of $M$ where we set $M_i = L_i$ for all $1 \leq i \leq l$.
By construction we have that $M_i / M_{i-1} \in \bigcup_{t \geq s} \P(t)$ for every $1 \leq i \leq n+l$. 
Hence $\Filt \left( \bigcup_{t \geq s} \P(t) \right)$ is closed under extensions.

We claim that $\T_s$ is closed under quotients.
Let $0 \to L \to M \to N \to 0$ be a short exact sequence and suppose that $M \in \T_s$.
By hypothesis there is a filtration 
$$0=M_0 \subset M_1 \subset \dots \subset M_n=M$$
of $M$ where $M_i / M_{i-1} \in \bigcup_{t > s} \P(t)$ for all $1 \leq i \leq n$.
This filtration of $M$ induces a filtration 
$$0=\frac{M_0 + L}{L}\subset \frac{M_1 + L}{L} \subset \dots\subset \frac{M_{n-1} + L}{L}  \subset \frac{M_n + L}{L} \cong N$$
of $N$ where $\frac{\frac{M_i + L}{L}}{\frac{M_{i-1} + L}{L}} \cong \frac{M_i + L}{M_{i-1}+L}$ is a quotient of $M_i / M_{i-1}$ for every $1 \leq i \leq n$.
From this filtration we obtain a filtration 
$$0=N_0 \subset N_1 \subset \dots \subset N_t=N$$
such that $N_j/N_{j-1}$ is a non-zero quotient for $M_i/M_{i-1} \in \P(r_i)$ for some $1 \leq i \leq n$. 
Then Lemma~\ref{lem:quotient} implies that $N_j/N_{j-1}\in \T_s$ for every $1 \leq j \leq t$. 
Then we can conclude that $\T_s$ is closed under quotients. 
A similar argument shows that $\overline{\F}_s$ is closed under extensions and subobjects. 

We claim that $\Hom_\A(X, Y)=0$ for all $X\in \T_s$ and $Y \in \overline{\F}_s$. 
Suppose that there exists a non-zero morphism $f \in \Hom_\A(X, Y)$ for some $X\in \T_s$ and $Y \in \overline{\F}_s$. 
Then the image $\im f$ of $f$ is both a quotient of $X$ and a subobject of $Y$. 
From the above we have that $\im f \in \T_s \cap \overline{\F}_s$.
This, in turn, imply the existence of a non-zero object $X_r \in \P(r) \cap \overline{\F}_s$ for some $r > s$. 
Then, by definition, we obtain a filtration 
$$0=(X_r)_0 \subset (X_r)_1 \subset \dots \subset (X_r)_t= X_r$$
such that every successive quotient is in $\bigcup_{z< s} \P(z)$. 
In particular we have that there exists a non-zero morphism $p : X_r \to X_r / (X_r)_{t-1}$ where $X_r / (X_r)_{t-1}$ is an object in $\P(z)$ where $z \leq s < r$, contradicting Definition~\ref{def:slicing}.1 and our claim follows.

We finish the proof by showing the existence of a canonical short exact sequence associated to every $M \in \A$.
Let $M$ be a non-zero object in $\A$.
Then by Definition~\ref{def:slicing}.2 there exists a filtration 
$$0=M_0 \subset M_1 \subset \dots \subset M_t=M$$
where $M_i / M_{i-1} \in \P(r_i)$ for every $1\leq i \leq t$ and $r_t < \dots < r_2 < r_1$. 
Set $k = \min_{1\leq i \leq t} \{r_i > s\}$. 
Then the short exact sequence 
$$ 0 \to M_k \to M \to M/M_k \to 0$$
is such that $M_k \in \T_s$ and $M/ M_k \in \overline{\F}_s$.
\end{proof}

\begin{lemma}\label{lem:chainsfrom}
Let $\P$ be a slicing of $\A$. 
Then $\P$ naturally induces two chains of torsion classes:
\begin{itemize}
    \item $\eta_\P = \left\{ \T_0 = \A, \T_s \text{ for every $s\in (0,1)$ and } \T_1 = \{0\}  \right\}$
    \item $\overline{\eta}_\P = \left\{ \T_0 = \A, \overline{\T}_s \text{ for every $s\in (0,1)$ and } \T_1 = \{0\}  \right\}$
\end{itemize}
where $\eta_\P, \overline{\eta}_\P \in \CT(\A)$.
Moreover $\bigcap\limits_{r<s} \T_r = \bigcap\limits_{r<s} \overline{\T}_r = \overline{\T}_s$ and $\bigcap\limits_{r>s} \overline{\F}_r = \bigcap\limits_{r>s} \F_r = \overline{\F}_s$ for every $s \in (0,1)$.
\end{lemma}

\begin{proof}
Take $r < s$.
The fact that $\T_s$, $\T_r$, $\overline{\T}_s$ and $\overline{\T}_r$ are torsion classes has been shown in Proposition~\ref{prop:torsasfilt}.
Then it is obvious that $\T_s \subseteq \T_r$ and $\overline{\T}_s \subseteq \overline{\T}_r$. 
Therefore $\eta_\P$ and $\overline{\eta}_\P$ as defined in the statement is a chain of torsion classes.
The fact that $\eta_\P, \overline{\eta}_\P \in \CT(\A)$ follows from the fact that the filtration in Definition~\ref{def:slicing} is finite for every object $M\in \A$.

Now we prove the moreover part of the statement.
We only show that 
$\overline{\T}_s = \bigcap_{r<s} \overline{\T}_r$ since the other cases are similar.
The fact that $\overline{\T}_s \subseteq \bigcap\limits_{s>t} \overline{\T}_t$ is a direct consequence of the that that $\eta_\P$ is a chain of torsion classes.
To prove the other direction note that, by definition of $\overline{\T}_s$, every object in the intersection of $\P(r)$ and $\overline{\T}_s$ is isomorphic to the zero object if $r < s$.
Now, let $M \in \bigcap\limits_{s>t} \overline{\T}_t$ be a non-zero object and consider the filtration 
$$0=M_0 \subset M_1 \subset \dots \subset M_n=M$$
of $M$ given by Definition \ref{def:slicing}.
Then, by definition, $M_n / M_{n-1}$ is a non-zero object in $\P(t)$ for some $t \in [0,1]$. 
Also $M_n / M_{n-1} \in \bigcap\limits_{s>t} \overline{\T}_t$ because $\bigcap\limits_{s>t} \overline{\T}_t$ is a torsion class. 
Therefore, the remark we did at the beginning of this paragraph implies that $M_n / M_{n-1} \in \P(t)$ where $t \geq s$. 
In particular, Definition \ref{def:slicing} implies that 
$$M \in \mbox{Filt} \left(\bigcup_{t \geq s} \P(t)\right)=\overline{\T}_s.$$
\end{proof}

Recall that for every chain of torsion classes $\eta \in \CT(\A)$ we define the set $\P_\eta$ of subcategories of $\A$ to be $\P_\eta = \{\P_{\eta}(t) : t \in [0,1]\}$.

\begin{theorem}\label{thm:eqslicingtors}
Let $\A$ be a abelian category. 
Then every chain of torsion classes $\eta$ in $\CT(\A)$ induces a slicing $\P_\eta$ in $\P(\A)$ and every slicing arises this way.
\end{theorem}

\begin{proof}
Let $\eta \in \CT(\A)$.
The fact that $\P_\eta$ is a slicing follows directly from Theorem \ref{thm:algHNfilt}  and Theorem \ref{thm:quasistables}.4.

Conversely, if $\P$ is a slicing of $\A$, $\P$ induces a chain of torsion classes $\eta_\P \in \CT(\A)$ by Lemma \ref{lem:chainsfrom}.
We claim that $\P_{\eta_\P}$ coincides with $\P$.
Then we need to calculate $\P_{\eta_\P}(s) \in \P_{\eta_\P}$ for every $s \in [0,1]$.
By definition we have that 
$$\P_{\eta_\P}(s)= \bigcap\limits_{s>t} \T_t\cap \bigcap\limits_{r>s}
\F_r.$$
Now, we apply Lemma \ref{lem:chainsfrom} to get that 
$$\bigcap\limits_{s>t} \T_t\cap \bigcap\limits_{r>s}
\F_r = \mbox{Filt} \left(\bigcup_{t \geq s} \P(t)\right) \cap \mbox{Filt} \left(\bigcup_{t \leq s} \P(t)\right)$$
Clearly, 
$$ \P(s) \subseteq \mbox{Filt} \left(\bigcup_{t \geq s} \P(t)\right) \cap \mbox{Filt} \left(\bigcup_{t \leq s} \P(t)\right):=\P_{\eta_\P}(s)$$
Let $M \in \P_{\eta_\P}(s)$ and consider its filtration 
$$0=M_0 \subset M_1 \subset \dots \subset M_n=M$$ 
given by Definition \ref{def:slicing}.2.
Since $M \in \mbox{Filt} \left(\bigcup\limits_{t \geq s} \P(t)\right)$ then
$M_k/M_{k-1} \in \P(r_k)$
where $r_k \geq s$ for every $k$.
Likewise, since $M \in \mbox{Filt} \left(\bigcup\limits_{t \leq s} \P(t) \right)$ we have that $M_k/M_{k-1} \in \P(r_k)$ where $r_k \leq s$ for every $k$.
Hence $M \in \mbox{Filt} (\P(s)) =\P(s)$. 
\end{proof}


\section{Maximal green sequences}\label{sec:MGS}

In this section we consider a particular type of chains of torsion classes, the so-called \textit{maximal green sequences}.
The concept of maximal green sequence was first introduced in \cite{Keller2011b} in the context of cluster algebras and its axiomatic study started in \cite{Brustle2014}. 
With the introduction of $\tau$-tilting theory \cite{AIR}, it has been shown that the combinatorics of cluster algebras were encoded in the homological algebra of $\mod A$, the category of finitely generated (right) modules over a finite-dimensional algebra $A$ over a field $k$. 
This allowed the formal study of maximal green sequences in any abelian category \cite{Ingalls2009,BSTpath}.
For a complete survey on the subject, we refer the reader to \cite{DK2019}.

\begin{definition}\label{def:MGS}
A \textit{maximal green sequence} in an abelian category $\A$ is a finite sequence of torsion classes 
$$ 0=\T_t\subsetneq \T_{t-1} \subsetneq \dots \subsetneq \T_{1}\subsetneq \T_0=\A$$
such that for all $i\in\{1, 2, \dots, t\}$, the existence of a torsion class $\T$ satisfying $\mathcal{T}_{i+1}\subseteq\mathcal{T}\subseteq\mathcal{T}_{i}$ implies $\mathcal{T}=\mathcal{T}_i$ or $\mathcal{T}=\mathcal{T}_{i+1}$.
\end{definition}

Since every maximal green sequence is in $\CT(\A)$, we obtain the following corollary as a particular case of Theorem~\ref{thm:algHNfilt}, which is a generalisation of \cite[Theorem 5.13]{Igusa2017}. 

\begin{corollary}\label{cor:MGS}
Let $\A$ be an abelian length category and $\eta$ be a maximal green sequence in $\A$. 
Then $\eta$ induces a Harder-Narasimhan filtration for every non-zero object $M$ of $\A$.
\end{corollary}

\begin{proof}
Let $\eta$ be a maximal green sequence as defined in Definition \ref{def:MGS}. 
Hence $\eta$ can be seen as a chain of torsion classes defined as $\T_x = \T_i$ if $x \in \left[\frac{i}{t}, \frac{i+1}{t} \right)$ for all $i \in \{1, \dots, t\}$.
Then $\eta\in\CT(\A)$ by Remark~\ref{rmk:etainCT(A)}.
The result follows from Theorem~\ref{thm:algHNfilt}.
\end{proof}

The previous result is valid for every abelian category. 
However, if we suppose that $\A=\mod A$ we are able to give a better description of the Harder-Narasimhan filtration induced by a maximal green sequence using $\tau$-tilting theory.
In this section $\tau$ denotes the Auslander-Reiten translation in $\mod A$.

\begin{definition}\cite[Definition 0.1 and 0.3]{AIR}\label{def:list}
Let $A$ be an algebra, $M$ an $A$-module and $P$ a projective $A$-module. The pair $(M,P)$ is said \textit{$\tau$-rigid} if:
	\begin{itemize}
		\item $\Hom_{A}(M,\tau M)=0$;
		\item $\Hom_{A}(P,M)=0$.
	\end{itemize}
Moreover, we say that $(M,P)$ is \textit{$\tau$-tilting} (or \textit{almost $\tau$-tilting}) if $|M|+|P|=n$ (or $|M|+|P|=n-1$, respectively), where $|M|$ denotes the number of non-isomorphic indecomposable direct summands of $M$.
\end{definition}

The implications of $\tau$-tilting theory in representation theory are multiple, some of them yet to be understood. 
For instance, it was already mentioned before the relation of this theory with cluster algebras.
From a more representation theoretic point of view, $\tau$-tilting theory has the fundamental property of describing all functorially finite torsion classes. 
An overview of this rich subject can be found in \cite{TreffingerSurvey}.
The exact definition of functorial finiteness is not necessary for the results of this paper and it will be skipped.

\begin{theorem}\cite[Theorem 2.7]{AIR}\cite[Theorem 5.10]{Auslander1981}
There is a well-defined function $\Phi: \mathrm{s\tau\text{-}rig} A \to \mathrm{f\text{-}tors}$ from $\tau$-rigid pairs to functorially finite torsion classes given by 
$$\Phi(M,P)=\Fac M:=\{X \in\mod A : M^{r} \to X \to 0 \text{ for some $r\in\mathbb{N}$}\}.$$
Moreover, $\Phi$ is a bijection if we restrict it to $\tau$-tilting pairs.
\end{theorem}

\begin{theorem}\label{thm:MGSforAlg}
Let $A$ be finite-dimensional algebra over a field $k$ and let 
$$\eta:= \qquad 0=\T_t\subsetneq \T_{t-1} \subsetneq \dots \subsetneq \T_{1}\subsetneq \T_0=\mod A$$
be a maximal green sequence in $\mod A$. 
Then every object $M\in \A$ admits a Harder-Narasimhan filtration by bricks.
That is a filtration 
$$M_0 \subset M_1 \subset \dots \subset M_n$$
and a set $\mathcal{S}_\eta=\{B_1, \dots, B_t\}$ of bricks in $\mod A$ such that the following conditions are satisfied:
\begin{enumerate}
\item $0 = M_0$ and $M_n=M$;
\item $M_k/M_{k-1} \in \Filt(B_{i_k})$ for some $i_k \in \{1, \dots, t\}$ for all $1\leq k \leq n$;
\item the indexes $i_k$ defined in the previous item are such that $i_k >i_{k'}$ if and only if $k<k'$.
 \item $\emph{Hom}_A (B_k, B_{k'})=0$ if $k<k'$.
\end{enumerate}
Moreover this filtration is unique up to isomorphism.
\end{theorem}

\begin{proof}
We want to show that the Harder-Narasimhan filtration of $M$ given in Corollary \ref{cor:MGS} has all the characteristics named in the statement.
It follows from \cite[Theorem 3.1]{DIJ} that every torsion class $\T_i$ in $\eta$ is generated by a $\tau$-tilting pair $(M_i,P_i)$, that is $\T_i= \Phi(M_i,P_i)= \Fac M_i$ for every $i$.
Also, it follows from \cite[Corollary~3.21]{BSTw&c} that $\P_i$ is a wide subcategory.
Moreover, \cite[Theorem 3.14]{BSTw&c} implies that every object $M$ in $\P_\eta(i)$ is filtered by a single brick $B_i \in \P_\eta(i)$.
Therefore is enough to define $\mathcal{S}_\eta$ as 
$$\mathcal{S}_\eta := \left\{B_i : \P_\eta(i)=\Filt(B_i) \text{ for all $1 \leq i \leq t$} \right\}$$
to show 1., 2. and 3.
Finally 4. follows directly from Theorem \ref{thm:quasistables}.4.
\end{proof}

It is well-known that the Grothendick group $K_0(A)$ of an algebra $A$ is isomorphic to $\mathbb{Z}^n$, for some natural number $n$, where the classes $[S(i)]$ of the simple modules $S(i)$ form a basis of $K_0(A)$.
It was recently shown \cite{Treffinger2019, SchrollTreffinger} that the classes $[B]$ in $K_0(A)$ of certain bricks $B$ in the module category are in bijection with the so-called \textit{$c$-vectors} of the algebra, a notion first introduced in \cite{Fomin2003} for cluster algebras and later adapted to module categories in \cite{Fu2017}. 
Based on the results of \cite{Treffinger2019, SchrollTreffinger}, Theorem \ref{thm:MGSforAlg} can be reformulated as follows, generalising \cite[Lemma 2.4]{Garver2018}.

\begin{corollary}
Let $A$ be an algebra, then every maximal green sequence 
$$\eta:= \qquad 0=\T_t\subsetneq \T_{t-1} \subsetneq \dots \subsetneq \T_{1}\subsetneq \T_0=\mod A$$ 
is determined by an ordered set $\{c_1, \dots, c_t\}$ of (positive) $c$-vectors.
\end{corollary}

\begin{proof}
It follows from Theorem \ref{thm:MGSforAlg} that every maximal green sequence $\eta$ induces a ordered set $\mathcal{S}_\eta$ of bricks in $\mod A$.
Then, consider the ordered set of vectors $\{[B_i] : B_i \in \mathcal{S}_\eta\}$ with the order induced by the order of $\mathcal{S}_\eta$.
It follows from \cite[Theorem 3.5]{Treffinger2019} that $\{[B_i] : B_i \in \mathcal{S}_\eta\}$ is a ordered set of (positive) $c$-vectors.

Now, suppose that we have two different maximal green sequences 
$$\eta:= \qquad 0=\T_t\subsetneq \T_{t-1} \subsetneq \dots \subsetneq \T_{1}\subsetneq \T_0=\mod A$$ 
and 
$$\eta':= \qquad 0=\T'_s\subsetneq \T'_{s-1} \subsetneq \dots \subsetneq \T'_{1}\subsetneq \T'_0=\mod A.$$ 
Then there is a first $i$ such that  $1 \leq i \leq t$, $1 \leq i \leq s$ and $\T_i \neq \T'_i$.
So, we have a brick $B_i \in \mathcal{S}_\eta$ and $B'_i \in \mathcal{S}_\eta'$ by Theorem \ref{thm:MGSforAlg}.
Moreover \cite[Theorem 3.5]{Treffinger2019} implies that $[B_i]\neq [B'_i]$.
Therefore every maximal green sequence is determined by their ordered set $\{c_1, \dots, c_t\}$ of $c$-vectors.
\end{proof}


\section{Hall algebras and wall-crossing formulas}\label{sec:wallcrossing}

In this section we consider $\A$ to be an essentially small abelian length category such that $\Hom_\A(X,Y)$ is a finite set for every pair of objects $X, Y$ of $\A$.
An example of such categories is the module category of a finite dimensional algebra over a finite field $k$.
In this section, we work with the completed Hall algebra $\wH(\A)$ of $\A$, whose elements are the formal series 
$$\sum_{{[M]\in\operatorname{Iso}(\A)}} a_M [M]$$
where $\operatorname{Iso}(\A)$ is the set of isomorphism classes of objects in $\A$, $[M]\in\operatorname{Iso}(\A)$ denotes the isomorphism class of $M\in \A$, and $a_{M}$ is an element of $k$.
The product in $\wH(\A)$ of two elements is determined by the product on its basis elements 
$$[L] \cdot [N]:= \sum c_{L N}^M [M]$$
where $c_{LN}^M$ is the number of subobjects $L'$ of $M$ which are isomorphic to $L$ such that $M / L'$ is isomorphic to $N$.
In other words,  $c_{LN}^M$ counts the number of ways we can write $M$ as an extension of $N$ by $L$, up to isomorphism.
The algebra $\wH(\A)$ is associative, noncommutative and with 
$$1_{\wH(\A)} = [0],$$
the isomorphism class of the zero object of $\A$, as its unit.
In $\wH(\A)$ there is a distinguished object $e_\A$, namely
$$e_{\A} := \sum_{[M]\in\operatorname{Iso}(\A)} [M].$$

Two different stability conditions on $\A$ can give rise to quite different Harder-Narasimhan filtrations for an arbitrary object $M$ of $\A$.
Despite this, it was shown in \cite{Reineke2003} that Harder-Narasimhan filtrations induced by every stability condition $\phi$ give rise to a factorisation of the element $e_{\A}$. 
Then, the equality between different factorisations of $e_\A$ induced by stability conditions are known under the name of \textit{wall-crossing formulas}. 

In this section we combine Reineke's ideas \cite{Reineke2003} with Theorem \ref{thm:algHNfilt} to show that wall-crossing formulas can be deduced directly from chains of torsion classes in $\A$, bypassing the construction for the suitable stability condition.
This section follows closely the exposition made in \cite[Section 4]{Bridgeland2018} for finitary Hall algebras, but the argument also works for more general Hall algebras.

Given a subcategory $\X$ of $\A$ we define the element $e_\X$ of $\wH(\A)$ associated to $\X$ to be 
$$e_\X := \sum_{[M]\in\operatorname{Iso}(\X)} [M]$$
Let $S$ be a totally ordered set and let $\{\X_s \colon 0\in\X_s\subseteq \A \text{ and } s \in S\}$ be a set of full subcategories of $\A$ indexed by $S$.
Immediately, we have a family $\{ e_{\X_s} \in \wH(\A) : s \in S\}$ of elements of $\wH(\A)$ indexed by $S$.
We write $\prod\limits_{s\in S}\limits^{\leftarrow} e_{\X_s}$ as a short hand of all the possible finite products of the classes of non-zero objects $M_{s_i} \in \X_{s_i}$ with decreasing index $s$.
This can be condensed in the following formula. 
$$\prod^{\leftarrow}_{s\in S} e_{\X_s} := \sum_{j \in \mathbb{N}} \left(\sum_{\substack{s_1 > \dots > s_j \\ [M_t]\in\operatorname{Iso}(\X_t)}} [M_{s_1}][M_{s_2}]\dots[M_{s_j}]\right)$$
Now we state the main result of this section.

\begin{theorem}\label{thm:wallcrossing}
Let $\A$ be an abelian category and let $\wH(\A)$ its completed Hall algebra. 
Then for every chain of torsion classes $\eta$ in $\CT(\A)$ the following equality holds
$$e_\A = \prod^{\leftarrow}_{t\in [0,1]} e_{\P_\eta(t)}$$
where $e_{\P_t}$ is the element associated to $\P_t$ for all $t \in [0,1]$.
\end{theorem}

\begin{proof}
By definition 
$$\prod^{\leftarrow}_{t\in [0,1]} e_{\P_\eta(t)} = \prod^{\leftarrow}_{t\in [0,1]} \left(\sum_{[M]\in\operatorname{Iso}(\P_\eta(t))} [M]\right)$$
where $\P_t$ is the category of $\eta$-quasisemistable objects of phase $t$.
$$\prod^{\leftarrow}_{t\in [0,1]} \left(\sum_{[M]\in\operatorname{Iso}( \P_\eta(t))} [M]\right) = [0] + \sum_{j \in \mathbb{N}} \left(\sum_{\substack{s_1 > \dots > s_j \\ [0]\neq[M_{s_i}]\in\operatorname{Iso}(\P_\eta(s_i))}} [M_{s_1}][M_{s_2}]\dots[M_{s_j}]\right)$$
This is a well-defined element of $\wH(\A)$, so we can write it as
$$[0]+\sum_{k \in \mathbb{N}} \left(\sum_{s_1 > \dots > s_k} [M_{s_1}][M_{s_2}]\dots[M_{s_k}]\right) = [0] + \sum_{[0]\neq[M]\in\operatorname{Iso}(\A)} a_M[M]$$
where $M$ correspond to the isomorphism classes of non-zero objects in $\A$ that can be filtered by objects in $\P_\eta$ with decreasing phase. 
By Theorem \ref{thm:algHNfilt} that correspond to the isomorphism class of every non-zero object $M \in \A$. 
Moreover, this filtration is unique up to isomorphism by Theorem \ref{thm:algHNfilt}.
Thus $a_M = 1$ for every $[M]\in\operatorname{Iso}(\A)$. 
In conclusion, we have that 
$$\prod_{t\in [0,1]}^{\leftarrow} e_{\P_\eta(t)} = [0] + \sum_{M \in \A} [M] = e_\A,$$
as claimed.
\end{proof}

From the previous result and Theorem\;\ref{thm:semistabchain}, we recover the classical wall-crossing formulas.

\begin{corollary}\cite[Theorem 5.1]{Reineke2003}
Let $\A$ be an abelian category and let $\wH(\A)$ its completed Hall algebra. 
Then for every stability function $\phi: Obj^*(\A) \to [0,1]$ we have that
$$e_\A = \prod^{\leftarrow}_{t\in [0,1]} e_{\P_\phi(t)}$$
where $e_{\P_\phi(t)}$ is the element associated to the subcategory 
$$\P_\phi(t) = \{M \in \A: M \text{ is $\phi$-semistable and } \phi(M)=t\} \cup \{0\}$$
for all $t \in [0,1]$. \qed
\end{corollary}

From Theorem \ref{thm:wallcrossing} we also recover the so-called \textit{torsion pair identity} of Hall algebras, generalising \cite[Proposition 6.4]{Bridgeland2016}.

\begin{corollary}
Let $\A$ be an abelian category and let $\wH(\A)$ its completed Hall algebra.
Then every torsion pair $(\T, \F)$ in $\A$ induces the following identity in $\wH(\A)$.
$$e_\A = e_\T e_\F$$

\end{corollary}

\begin{proof}
In example \ref{ex:torsionpair} we showed that every torsion pair in $\A$ induces a chain of torsion classes $\eta_{(r_1, r_2)}$, where $r_1, r_2 \in [0,1]$ and $r_1 < r_2$. 
Then, it is enough to take $t \in (r_1, r_2)$ and apply Theorem \ref{thm:wallcrossing} for $\eta_{(r_1, r_2)}$.
\end{proof}

\paragraph{Acknowledgements} 
The author is thankful to the people of the D\'epartement de Math\'ematiques de l'Universit\'e de Sherbrooke, where the main ideas of this project were first written.
He is also thankful for the referee for their valuable work and insightful comments that helped improve the readability of this paper.

\paragraph{Funding} Throughout the course of this work the author was funded by the European Union’s Horizon 2020 research and innovation programme under the Marie Sklodowska-Curie grant agreement No 893654, by the Deutsche Forschungsgemeinschaft (DFG, German Research Foundation) under Germany's Excellence Strategy Programme -- EXC-2047/1 -- 390685813, and by the EPSRC through the Early Career Fellowship, EP/P016294/1.

\def\cprime{$'$} \def\cprime{$'$}

\bigskip
\noindent
\textsf{[Treffinger]} \textsc{Instituto de Investigaciones Matemáticas ''Luis Santaló" \\
UBA-CONICET\\
Pabellón I, Ciudad Universitaria\\
(1428) Buenos Aires, Argentina}\\
{\scriptsize
URL: {\tt https://sites.google.com/view/hipolitotreffinger}\\
e-mail: {\tt htreffinger@dm.uba.ar}
}

\end{document}